\documentclass[reqno,11pt]{amsart}
\calclayout
\usepackage{mathtools,accents}
\usepackage[unicode,hidelinks]{hyperref}
\newcommand\ml{\operatorname{\mathsf{L}}}
\newcommand\mr{\operatorname{\mathsf{R}}}
\newcommand\ac{\operatorname{\mathsf{a\!c}}}
\newcommand\ad{\operatorname{\mathsf{a\!d}}}
\newtheorem{theorem}{Theorem}
\newtheorem{corollary}[theorem]{Corollary}
\newtheorem{lemma}[theorem]{Lemma}
\newtheorem{conjecture}[theorem]{Conjecture}
\theoremstyle{definition}
\newtheorem{definition}[theorem]{Definition}
\newtheorem{example}[theorem]{Example}
\numberwithin{theorem}{section}
\numberwithin{equation}{section}

\title[Commutator calculus]{Commutator calculus\\and~symbolic differentiation of~matrix functions}
\author{\href{https://orcid.org/0000-0002-4031-9812}{Michal Bathory}}
\subjclass{15A16, 47A56, 47B15, 47B47, 76A10}
\keywords{matrix functional calculus, matrix exponential, Fr\'{e}chet derivative, elasticity}
\address{Faculty of Mathematics and Physics\\Charles University\\Sokolovsk\'a 83\\186\;75 Prague\\Czech Republic}
\email{bathory@karlin.mff.cuni.cz}
\thanks{Michal Bathory was supported by the project No.~20-11027X financed by GA\v{C}R}

\begin{document}
\begin{abstract}
	We propose a functional calculus which allows one to apply functions to the matrix anti-commutator/commutator operator. The calculus is introduced in a straightforward manner if the operators act on symmetric matrices, and it leads to a coordinate-free version of Dalecki\u{\i}--Kre\u{\i}n formula. In this sense, the proposed calculus provides symbolic formulae for the derivatives of matrix-valued functions that are explicit and easy to use. We discuss several applications of the newly introduced calculus in continuum mechanics (Hencky logarithmic strain, objective rates, spin tensors, and viscoelastic fluids) and in the theory of partial differential equations.
\end{abstract}
\maketitle
\section{Motivation}

This work is devoted to developing a practical functional calculus for a class of linear operators on matrices. In various areas of physics, one inherently encounters \emph{tensor}-valued quantities and perturbations thereof, leading to non-commutativity issues with respect to the matrix multiplication, or composition of operators in general. The usual way to address this is via a power series expansion in the matrix commutator, which, however, often leads to unwieldy expressions of limited applicability. Here we present an alternative approach that allows us to handle these expansions symbolically with ease, using a functional calculus that is applied directly to the commutator operator and to some related basic operators on matrices.

For a fixed matrix $\mathbf{G}\in\mathbb{R}^{d\times d}$, let us consider the quartet of linear operators $\ml_{\mathbf{G}}$, $\mr_{\mathbf{G}}$, $\ac_{\mathbf{G}}$, $\ad_{\mathbf{G}}\in\mathcal{L}(\mathbb{R}^{d\times d})$, $d\geq 1$, defined as
\begin{equation}\label{qua}
    \ml_{\mathbf{G}}\mathbf{X}:=\mathbf{G}\mathbf{X},\;\;\mr_{\mathbf{G}}\mathbf{X}:=\mathbf{X}\mathbf{G},\;\;\ac_{\mathbf{G}}\mathbf{X}:=\frac{\mathbf{G}\mathbf{X}+\mathbf{X}\mathbf{G}}2,\;\;\ad_{\mathbf{G}}\mathbf{X}:=\frac{\mathbf{G}\mathbf{X}-\mathbf{X}\mathbf{G}}2,
\end{equation}
for any $\mathbf{X}\in\mathbb{R}^{d\times d}$. These four operators represent the left/right matrix multiplication by $\mathbf{G}$ and the (halved) anti-commutator/commutator operators, respectively. For simplicity, in this work the matrix $\mathbf{G}$ is usually symmetric, denoted as $\mathbf{G}\in\mathbb{R}^{d\times d}_{\rm sym}$, although most of our conclusions also extend to the asymmetric case. Note that the operators $\ml_{\mathbf{G}},\mr_{\mathbf{G}},\ac_{\mathbf{G}},\ad_{\mathbf{G}}$ defined by \eqref{qua} can be seen as fourth-order tensors whose components derive from the underlying matrix $\mathbf{G}$ and its spectral decomposition, which is the only place where the choice of coordinates matters. Otherwise, we shall adhere to the corresponding coordinate-free operator representations, taking a general functional-analytic approach. The need to apply a~function to the operators in \eqref{qua} stems from the problem of finding symbolic matrix derivatives of classical matrix functions such as exponential or logarithm. For instance, recall that~a classical Lie group formula (cf.~\cite[Th.~5.4]{Hall_2015} and \cite[Sect.~1.2., Th.~5]{Rossmann2002}) reads as
\begin{equation}\label{lieexp1}
	\frac{de^{\mathbf{G}}}{d\mathbf{G}}\mathbf{X}=e^{\mathbf{G}}\Big(\frac{1-e^{-2\ad_{\mathbf{G}}}}{2\ad_{\mathbf{G}}}\mathbf{X}\Big)=\ml_{e^{\mathbf{G}}}\eta(-2\ad_{\mathbf{G}})\mathbf{X},\quad\mathbf{X}\in\mathbb{R}^{d\times d},
\end{equation}
where
\begin{equation}\label{eta}
	\eta(z):=\frac{e^z-1}z=\sum_{n=0}^{\infty}\frac{z^n}{(n+1)!},\quad z\in\mathbb{C},
\end{equation}
and $\eta(-2\ad_{\mathbf{G}})\in\mathcal{L}(\mathbb{R}^{d\times d})$ is interpreted by formally substituting $z=-2\ad_{\mathbf{G}}$ into the above series and understanding $\ad_{\mathbf{G}}^n$ as the $n$-fold composition of $\ad_{\mathbf{G}}\in\mathcal{L}(\mathbb{R}^{d\times d})$ with itself. The right-hand side of \eqref{lieexp1} consists of the composition of two (commuting) fourth-order tensors $\ml_{e^{\mathbf{G}}}$ and $\eta(-2\ad_{\mathbf{G}})$, which are then applied to the matrix $\mathbf{X}$. We shall see later that similar decompositions can be found for derivatives of other matrix functions, where one of the operators always arises as a certain function of the commutator. However, not only that interpreting $\eta(-2\ad_{\mathbf{G}})$ through the power series \eqref{eta} and evaluating the nested commutators is cumbersome, but also this approach is limited by the convergence radius of the series being applied. For example, note that formally inverting \eqref{lieexp1} in order to obtain a~similar formula for the derivative of the matrix logarithm yields
\begin{equation}\label{dlog1}
	\frac{d\log\mathbf{A}}{d\mathbf{A}}\mathbf{Y}=\frac{2\ad_{\log\mathbf{A}}}{1-e^{-2\ad_{\log\mathbf{A}}}}(\mathbf{A}^{-1}\mathbf{Y}),\quad\mathbf{Y}\in\mathbb{R}^{d\times d},\;\mathbf{A}\in\mathbb{R}^{d\times d}_{\rm sym}\text{ positive definite}.
\end{equation}
Although we shall see later that \eqref{dlog1} is indeed true \emph{without any smallness restrictions} on~$\mathbf{A}-\mathbf{I}$, this result is beyond the reach of the power-series formalism since~$\eta(\pm2k\pi i)=0$, $k\in\mathbb{N}$, and thus~$1/\eta$ does not have an~everywhere convergent Taylor series, and the other types of expansions inevitably bring other difficulties.

The expressions~\eqref{lieexp1} and~\eqref{dlog1} show that it is desirable to look for other ways of interpreting the symbol~$f(\ad_{\mathbf{G}})$, where~$f$ is not necessarily an~entire function. Although one can interpret $f(\ad_{\mathbf{G}})$ using general tools of functional analysis applied to the (self-adjoint) operator $\ad_{\mathbf{G}}$, it turns out that in this case there exists an elegant approach that is particularly easy to work with and derives explicitly from the spectral decomposition of $\mathbf{G}$. Moreover, the functional calculus developed in this way not only gives meaning to expressions of type $f(\ad_{\mathbf{G}})$, but also provides a tool for symbolic differentiation of matrix-valued functions due to a formal similarity with the Dalecki\u{\i}--Kre\u{\i}n formula. For example, it turns out that the Fr\'echet derivative of a matrix function $f$ can be conveniently expressed in terms of the matrix left and right multiplication operators as
\begin{equation}
	\frac{df(\mathbf{G})}{d\mathbf{G}}=\frac{f(\ml_{\mathbf{G}})-f(\mr_{\mathbf{G}})}{\ml_{\mathbf{G}}-\mr_{\mathbf{G}}}.
\end{equation}
The right-hand side of this identity is understood as an evaluation of the function $F(x,y)=\frac{f(x)-f(y)}{x-y}$ at $x=\ml_{\mathbf{G}}$ and $y=\mr_{\mathbf{G}}$ (taking an appropriate limit if necessary), resulting in another fourth-order tensor $F(\ml_{\mathbf{G}},\mr_{\mathbf{G}})\in\mathcal{L}(\mathbb{R}^{d\times d})$, see Theorem~\ref{Tcons} below for details. The strength of such symbolic representations then allows one to rigorously justify some important identities that have been so far based only on formal manipulation with power series, cf.~\cite{lograte2025}, and to prove various qualitative properties of matrix-valued functions, such as monotonicity, which require one to deal with the differentiation of matrix-valued functions. Moreover, in some important cases (cf.~Theorem~\ref{TCayley}), we show how to convert symbolic expressions $f(\ad_{\mathbf{G}})$ into usual matrix multiplications, leading to formulae that can be easily evaluated without actually needing to compute the spectral decomposition of $\mathbf{G}$, see also Example~\ref{example} below.

\section{Functions applied to the commutator}

In this section, we introduce the calculus, see Definition~\ref{Def}, and study its elementary properties.

For simplicity, let us assume that~$\mathbf{G}\in\mathbb{R}^{d\times d}_{\rm sym}$ is symmetric, which allows us to write~$\mathbf{G}$ in a~form
\begin{equation}\label{Schur}
	\mathbf{G}=\mathbf{Q}\mathbf{g}\mathbf{Q}^{\mathsf{T}},\quad\text{where}\quad\mathbf{g}=\operatorname{diag}(g_i)_{i=1}^d\quad\text{and}\quad\mathbf{Q}^{-1}=\mathbf{Q}^{\mathsf{T}}
\end{equation}
by Schur spectral decomposition and $g_i$, $i=1,\ldots,d$ are real eigenvalues of $\mathbf{G}$. We shall also make use of the Hadamard/Schur matrix product~$\mathbf{X}\odot \mathbf{Y}$, defined as
\begin{equation*}
	\mathbf{X}\odot\mathbf{Y}:=[\mathbf{X}_{ij}\mathbf{Y}_{ij}].
\end{equation*}
Hereafter, we use the brackets $[m_{ij}]$ to denote a~square matrix of size $d\times d$ with elements $m_{ij}$, $i,j=1,\ldots,d$. Using the spectral decomposition of $\mathbf{G}$, we now give an~elementary definition of a~function applied to any of the operators introduced in~\eqref{qua}.

\begin{definition}\label{Def}
	Let~$\mathbf{G}\in\mathbb{R}^{d\times d}_{\rm sym}$ and let~$\mathbf{g}$ and~$\mathbf{Q}$ be the matrices from \eqref{Schur}. Further, let~$U\subset\mathbb{R}$ be a~neighborhood of the spectrum of~$\mathbf{G}$ and~$f:U\times U\to\mathbb{R}$. Then, for any~$\mathbf{X}\in\mathbb{R}^{d\times d}$, we define
	\begin{equation}\label{had}
		f(\ml_{\mathbf{G}},\mr_{\mathbf{G}})\mathbf{X}:=\mathbf{Q}\Big([f(g_i,g_j)]\odot(\mathbf{Q}^{\mathsf{T}}\mathbf{X}\mathbf{Q})\Big)\mathbf{Q}^{\mathsf{T}}.
	\end{equation}
	In this definition, if~$f$ has a~removable discontinuity at~$(g_i,g_j)$ for some~$1\leq i,j\leq d$, then we identify
	\begin{equation}\label{conv}
		f(g_i,g_j):=\lim_{(x,y)\to(g_i,g_j)}f(x,y).
	\end{equation}
\end{definition}

Identity \eqref{had} defines a~linear operator~$f(\ml_{\mathbf{G}},\mr_{\mathbf{G}})$ on~$\mathbb{R}^{d\times d}$, i.e., $f(\ml_{\mathbf{G}},\mr_{\mathbf{G}})\in\mathcal{L}(\mathbb{R}^{d\times d})$, and we often suppress the argument~$\mathbf{X}$ when possible. The above definition covers functions of the commutator: If the function~$h$ is defined in a~neighborhood of the points~$\frac{g_i-g_j}2$, $i,j=1,\ldots,d$, (which are the eigenvalues of~$\ad_{\mathbf{G}}$, cf.~\cite[Sect.~1.2.,~Lemma~8]{Rossmann2002}), then
\begin{equation}\label{comap}
	h(\ad_{\mathbf{G}})\mathbf{X}=h\Big(\frac{\ml_{\mathbf{G}}-\mr_{\mathbf{G}}}2\Big)\mathbf{X}=\mathbf{Q}\Big(\Big[h\Big(\frac{g_i-g_j}2\Big)\Big]\odot(\mathbf{Q}^{\mathsf{T}}\mathbf{X}\mathbf{Q})\Big)\mathbf{Q}^{\mathsf{T}},
\end{equation}
and similarly for the anti-commutator~$\ac_{\mathbf{G}}$, see also~\cite[Def.~1]{lograte2025}.

The purpose of identification \eqref{conv} is only to treat seamlessly functions with removable discontinuities, such as~$\eta$ and~$1/\eta$, avoiding the conditional definitions of these functions. In other words, we replace the functions with removable discontinuities by their continuous versions as we have no use of the former. We shall see later that this convention also allows for a~very concise representation of matrix derivatives, where a special case of~\eqref{had} is known as the Dalecki\u{\i}--Kre\u{\i}n formula, see \cite{Krein1956} and \cite[Sect.~4.2]{Najfeld_1995}. In this respect, the merit of Definition~\ref{Def} and of this work in general is that one can apply virtually the same formula to define the functions of the commutator as for the matrix derivatives, leading eventually to a~comprehensive and elegant calculus that avoids the problems with convergence of power series.

\subsection{Rules of the commutator calculus}

The Definition~\ref{Def} can only make sense if it is consistent with the interpretation of the power series, if it is independent of the choice of~$\mathbf{Q}$, and if it constructs a~homomorphism between certain fourth order tensors and real functions of two variables. The following theorem confirms that this is indeed the case and lists several other useful properties.

\begin{theorem}\label{Lrep}
	Let~$\mathbf{G}$,~$\mathbf{g}$,~$\mathbf{Q}$ and~$U$ be as in Definition~\ref{Def} and let~$f,f_1,f_2:U\times U\to\mathbb{R}$.
	\begin{itemize}
		\item[\rm(i)]  {\rm (Consistency)}
		Let $p:\mathbb{R}\times\mathbb{R}\to\mathbb{R}$ be the polynomial
		\begin{equation*}
			p(x,y)=\sum_{m=0}^M\sum_{n=0}^Np_{mn}x^my^n
		\end{equation*}
		for some $M,N\geq0$. Then
		\begin{equation*}
			p(\ml_{\mathbf{G}},\mr_{\mathbf{G}})=\sum_{m=0}^M\sum_{n=0}^Np_{mn}\ml_{\mathbf{G}}^m\mr_{\mathbf{G}}^n.
		\end{equation*}
		
		\item[\rm(ii)] {\rm(Commutativity)} If $\mathbf{H}\in\mathbb{R}^{d\times d}_{\rm sym}$ commutes with $\mathbf{G}$, $V\subset\mathbb{R}$ is a~neighbourhood of the spectrum of $\mathbf{H}$ and $\varphi:V\times V\to\mathbb{R}$, then
		\begin{equation}\label{com}
			f(\ml_{\mathbf{G}},\mr_{\mathbf{G}})\varphi(\ml_{\mathbf{H}},\mr_{\mathbf{H}})=\varphi(\ml_{\mathbf{H}},\mr_{\mathbf{H}})f(\ml_{\mathbf{G}},\mr_{\mathbf{G}}).
		\end{equation}
		
		\item[\rm(iii)] {\rm (Addition \& multiplication)} If $s,p:U\times U\to\mathbb{R}$ are such that
        \begin{equation*}
            s(x,y)=f_1(x,y)+f_2(x,y)\quad\text{and}\quad p(x,y)=f_1(x,y)f_2(x,y)
        \end{equation*}
        for all $x,y\in U$, then
		\begin{equation*}
			s(\ml_{\mathbf{G}},\mr_{\mathbf{G}})=f_1(\ml_{\mathbf{G}},\mr_{\mathbf{G}})+f_2(\ml_{\mathbf{G}},\mr_{\mathbf{G}})
		\end{equation*}
		and
		\begin{equation*}
			p(\ml_{\mathbf{G}},\mr_{\mathbf{G}})=f_1(\ml_{\mathbf{G}},\mr_{\mathbf{G}})f_2(\ml_{\mathbf{G}},\mr_{\mathbf{G}}).
		\end{equation*}
		
		\item[\rm(iv)] {\rm (Division)} If $f>0$ on $U\times U$, then $f(\ml_{\mathbf{G}},\mr_{\mathbf{G}})$ is bijective on $\mathbb{R}^{d\times d}$ and the inverse mapping satisfies
		\begin{equation*}
			(f(\ml_{\mathbf{G}},\mr_{\mathbf{G}}))^{-1}=(1/f)(\ml_{\mathbf{G}},\mr_{\mathbf{G}}).
		\end{equation*}	
		
		\item[\rm(v)] {\rm (Transposition)} For every $\mathbf{X}\in\mathbb{R}^{d\times d}$, there holds
		\begin{equation*}
			\big(f(\ml_{\mathbf{G}},\mr_{\mathbf{G}})\mathbf{X}\big)^{\mathsf{T}}=f(\mr_{\mathbf{G}},\ml_{\mathbf{G}})\mathbf{X}^{\mathsf{T}}.
		\end{equation*}
		
		\item[\rm(vi)] {\rm (Symmetry)} For all $\mathbf{X},\mathbf{Y}\in\mathbb{R}^{d\times d}$, one has
		\begin{align*}
			f(\ml_{\mathbf{G}},\mr_{\mathbf{G}})\mathbf{X}\cdot\mathbf{Y}&=\mathbf{X}\cdot f(\ml_{\mathbf{G}},\mr_{\mathbf{G}})\mathbf{Y},
		\end{align*}
		where the inner product is defined by $\mathbf{A}\cdot\mathbf{B}:=\sum_{i=1}^d\sum_{j=1}^d\mathbf{A}_{ij}\mathbf{B}_{ij}$ for any $\mathbf{A},\mathbf{B}\in\mathbb{R}^{d\times d}$.
		
		\item[\rm(vii)] {\rm (Norm)} There holds
		\begin{equation}\label{normf}
			\sup_{|\mathbf{X}|=1}|f(\ml_{\mathbf{G}},\mr_{\mathbf{G}})\mathbf{X}|\leq\big|[f(g_i,g_j)]\big|,
		\end{equation}
		where $|\mathbf{A}|:=\sqrt{\mathbf{A}\cdot\mathbf{A}}$, $\mathbf{A}\in\mathbb{R}^{d\times d}$, is the Frobenius matrix norm.
		
		\item[\rm(viii)] {\rm (Representation)} For any $\mathbf{X}\in\mathbb{R}^{d\times d}$, there holds 
		\begin{equation}\label{repreq}
			f(\ml_{\mathbf{G}},\mr_{\mathbf{G}})\mathbf{X}=\sum_{p,r=1}^d(\mathbf{J}_{\mathbf{G}}^f)_{pr}\mathbf{G}^{p-1}\mathbf{X}\mathbf{G}^{r-1},
		\end{equation}
		where $\mathbf{J}_{\mathbf{G}}^f\in\mathbb{R}^{d\times d}$ solves the linear system
		\begin{equation}\label{lp}
			\mathbf{V}_{\mathbf{G}}\mathbf{J}_{\mathbf{G}}^f\mathbf{V}_{\mathbf{G}}^{\mathsf{T}}=[f(g_i,g_j)]
		\end{equation}
		with the Vandermonde matrix
		\begin{equation*}
			\mathbf{V}_{\mathbf{G}}:=[g_i^{j-1}].
		\end{equation*}
		
		\item[\rm(ix)] {\rm (Canonical case)} If $\varphi:U\to\mathbb{R}$, then
		\begin{equation}\label{funcm}
			\varphi(\ml_{\mathbf{G}})=\ml_{\varphi(\mathbf{G})}\quad\text{and}\quad\varphi(\mr_{\mathbf{G}})=\mr_{\varphi(\mathbf{G})},
		\end{equation}
		where the matrix $\varphi(\mathbf{G})$ is defined as
		\begin{equation}\label{matf}
			\varphi(\mathbf{G})=\mathbf{Q}\varphi(\mathbf{g})\mathbf{Q}^{\mathsf{T}}
		\end{equation}
		and $\varphi(\mathbf{g})=\operatorname{diag}(\varphi\big(g_i)\big)_{i=1}^d$.
	\end{itemize}
\end{theorem}

Using Definition~\ref{Def}, the proof of the above theorem is straightforward and is deferred to Section~\ref{Sproof}, as well as most of the proofs in this work.

As a~consequence of part (i) and homomorphism properties (ii) and (iii), our definition of the symbol $f(\ml_{\mathbf{G}},\mr_{\mathbf{G}})$ agrees with the power series approach. This also shows that $f(\ml_{\mathbf{G}},\mr_{\mathbf{G}})$ must be independent of the choice of the matrix $\mathbf{Q}$. Property (iv) is crucial; on the one hand, it gives a~proper sense to \eqref{dlog1} and, on the other hand, we use it in the applications below to simplify complex problems that involve functions of the matrix commutator. Likewise, properties (v), (vi), (vii), and (ix) can often be utilized in applications. The representation formula in (viii) can be seen as a~variant of the Cayley--Hamilton theorem for the fourth-order tensors. It gives a~possible recipe how to compute $f(\ml_{\mathbf{G}},\mr_{\mathbf{G}})$ without the need to determine the eigenspaces of $\mathbf{G}$. However, this still involves solving the system \eqref{lp}. Completely explicit formulae for some useful particular cases are provided in the next subsection.

\subsection{Explicit formulae of function of the commutator in two and three dimensions}

Here we describe an~explicit solution to \eqref{lp} in dimensions $d\leq3$ and for the matrix commutator $\ad_{\mathbf{G}}$, which turns out to be the most important case in subsequent applications.

\begin{theorem}\label{TCayley}
	Let $\mathbf{G},\mathbf{g}\in\mathbb{R}^{d\times d}_{\rm sym}$ and $\{g_i\}_{i=1}^d$, be as in \eqref{Schur}. Let $f$ be defined in a~neighborhood of the set $\{\frac{g_i-g_j}2\}_{i,j=1}^d$ and consider its decomposition $f=f_0+f_{\rm odd}+f_{\rm even}$, where
	\begin{equation*}
		f_0:= f(0),\quad f_{\rm odd}(x):=\frac{f(x)-f(-x)}2\quad\text{and}\quad f_{\rm even}(x):=\frac{f(x)+f(-x)}2-f(0).
	\end{equation*}
	\begin{itemize}
		\item[\rm(i)] Let $d=1$. Then
		\begin{equation}
			f(\ad_{\mathbf{G}})\mathbf{X}=f_0\mathbf{X}.\label{cayley1d}
		\end{equation}
		\item[\rm(ii)] Let $d=2$. If $g_1\neq g_2$, then
		\begin{align}
			f(\ad_{\mathbf{G}})\mathbf{X}&=f_0\mathbf{X}+\frac{f_{\rm odd}(\frac{g_1-g_2}2)}{g_1-g_2}\big(\mathbf{G}\mathbf{X}-\mathbf{X}\mathbf{G}\big)\nonumber\\
			&\quad+\frac{f_{\rm even}(\frac{g_1-g_2}2)}{(g_1-g_2)^2}\big(-2g_1g_2\,\mathbf{X}+(g_1+g_2)(\mathbf{G}\mathbf{X}+\mathbf{X}\mathbf{G})-2\,\mathbf{G}\mathbf{X}\mathbf{G}\big),\label{cayley2d}
		\end{align}
		else \eqref{cayley1d} holds.
		\item[\rm(iii)] Let $d=3$. If $g_1\neq g_2\neq g_3\neq g_1$, then
		\begin{align}
			f(\ad_{\mathbf{G}})\mathbf{X}&=-K_2(\mathbf{G}\mathbf{X}-\mathbf{X}\mathbf{G})+K_1(\mathbf{G}^2\mathbf{X}-\mathbf{X}\mathbf{G}^2)-K_0(\mathbf{G}^2\mathbf{X}\mathbf{G}-\mathbf{G}\mathbf{X}\mathbf{G}^2)\nonumber\\
			&\quad+(f_0+2J_3L_1)\mathbf{X}-(J_2L_1+J_3L_0)(\mathbf{G}\mathbf{X}+\mathbf{X}\mathbf{G})\nonumber\\
			&\quad+2(L_2+J_2L_0)\,\mathbf{G}\mathbf{X}\mathbf{G}+(J_1L_1-L_2)(\mathbf{G}^2\mathbf{X}+\mathbf{X}\mathbf{G}^2)\nonumber\\
			&\quad-(L_1+J_1L_0)(\mathbf{G}^2\mathbf{X}\mathbf{G}+\mathbf{G}\mathbf{X}\mathbf{G}^2)+2L_0\,\mathbf{G}^2\mathbf{X}\mathbf{G}^2,\label{cayley3d}
		\end{align}
		where the invariants are defined by
		\begin{align*}
			J_1&=g_1+g_2+g_3,\quad
			J_2=g_1g_2+g_2g_3+g_3g_1,\quad
			J_3=g_1g_2g_3,\\
			K_n&=\frac{g_1^nf_{\rm odd}(\frac{g_2-g_3}2)+g_2^nf_{\rm odd}(\frac{g_3-g_1}2)+g_3^nf_{\rm odd}(\frac{g_1-g_2}2)}{(g_1-g_2)(g_2-g_3)(g_3-g_1)},\\
			L_n&=\frac{g_1^n\cfrac{f_{\rm even}(\frac{g_2-g_3}2)}{g_2-g_3}+g_2^n\cfrac{f_{\rm even}(\frac{g_3-g_1}2)}{g_3-g_1}+g_3^n\cfrac{f_{\rm even}(\frac{g_1-g_2}2)}{g_1-g_2}}{(g_1-g_2)(g_2-g_3)(g_3-g_1)},
		\end{align*}
		for $n=0,1,2$. In case of only two distinct eigenvalues, formula \eqref{cayley2d} holds instead, relabelling the indices if $g_1=g_2$. If all eigenvalues coincide, formula \eqref{cayley1d} takes place.
	\end{itemize}
\end{theorem}

The right-hand sides of \eqref{cayley2d} and \eqref{cayley3d} can be, of course, written again only in terms of $\ad_{\mathbf{G}}$, $\ac_{\mathbf{G}}$ and their compositions; however, this does not simplify the result.

In~the next section, we shall see that the derivatives of the classical matrix functions lead to expressions involving $f(\ad_{\mathbf{G}})$ for certain functions $f$. Theorem~\ref{TCayley} can be used to transform these expressions into a~usual form, without relying on Definiton~\ref{Def}. In~Example~\ref{example} below, we demonstrate this procedure. In addition to that, avoiding the determination of the matrix $\mathbf{Q}$ can be crucial in those applications where $\mathbf{G}$ itself is a function of some other variables and one has to investigate properties of the mapping $\mathbf{G}\mapsto f(\ad_{\mathbf{G}})$ such as continuity. Furthermore, the absence of eigenspaces of $\mathbf{G}$ in Theorem~\ref{TCayley} indicates that, using a~more general, and likely more technical, approach, one could define $f(\ml_{\mathbf{G}},\mr_{\mathbf{G}})$ also for non-symmetric square matrices $\mathbf{G}$. In fact, it seems that formulae \eqref{cayley2d} and \eqref{cayley3d} continue to hold in that case, although this is not proved here. We shall follow up on this observation in the next section (cf.~Conjecture~\ref{Conj} below), where we address the derivatives of the matrix functions.

\section{Symbolic differentiation of matrix functions}

In this work, by matrix functions we mean those that can be defined by the spectral decomposition as in~\eqref{matf}. For matrix derivatives, there are several possible notions. The following theorem puts forward a~unifying approach.

\begin{theorem}\label{Tcons}
	Let $f:\mathbb{R}\to\mathbb{R}$ and let $\mathbf{G}\mapsto f(\mathbf{G})$ be the corresponding matrix function. Let us formally define the operator $\frac{df(\mathbf{G})}{d\mathbf{G}}\in\mathcal{L}(\mathbb{R}^{d\times d})$ by
	\begin{equation}\label{dfG}
		\frac{df(\mathbf{G})}{d\mathbf{G}}:=\frac12\int_{-1}^1f'(\ac_{\mathbf{G}}+\,s\ad_{\mathbf{G}})ds,
	\end{equation}
	or, equivalently,
	\begin{equation}\label{dfG2}
		\frac{df(\mathbf{G})}{d\mathbf{G}}:=\int_0^1f'\big((1-t)\ml_{\mathbf{G}}+\,t\mr_{\mathbf{G}}\big)dt,
	\end{equation}
	whenever the expressions on the right-hand side exist.
	
	\begin{itemize}
		\item[\rm(A)] If $\mathbf{G}\in\mathbb{R}^{d\times d}_{\rm sym}$ and $f\in\mathcal{C}^1(I;\mathbb{R})$ for some interval $I\subset\mathbb{R}$ containing the spectrum of $\mathbf{G}$, then $\frac{df(\mathbf{G})}{d\mathbf{G}}$ fulfills the Dalecki\u{\i}--Kre\u{\i}n \cite{Krein1956} formula
		\begin{equation}\label{KD}
			\frac{df(\mathbf{G})}{d\mathbf{G}}\mathbf{X}=\mathbf{Q}\Big(\Big[\lim_{(x,y)\to(g_i,g_j)}\frac{f(x)-f(y)}{x-y}\Big]\odot(\mathbf{Q}^{\mathsf{T}}\mathbf{X}\mathbf{Q})\Big)\mathbf{Q}^{\mathsf{T}}
		\end{equation}
		for all $\mathbf{X}\in\mathbb{R}^{d\times d}$. Furthermore, there holds
		\begin{equation}\label{dfG0}
			\frac{df(\mathbf{G})}{d\mathbf{G}}=\frac{f(\ml_{\mathbf{G}})-f(\mr_{\mathbf{G}})}{\ml_{\mathbf{G}}-\mr_{\mathbf{G}}},
		\end{equation}
		where the right-hand side is interpreted using Definition~\ref{Def}.
		
		\item[\rm(B)] If $\mathbf{G}\in\mathbb{R}^{d\times d}$ and $f$ is given by an~everywhere convergent power series, then
		\begin{equation*}
			\frac{df(\mathbf{G})}{d\mathbf{G}}\mathbf{X}=\lim_{s\to0}\frac{f(\mathbf{G}+s\mathbf{X})-f(\mathbf{G})}{s}
		\end{equation*}
		for all $\mathbf{X}\in\mathbb{R}^{d\times d}$.
	\end{itemize}
	Hence, in both cases, the operator $\frac{df(\mathbf{G})}{d\mathbf{G}}$ is the Fr\'echet derivative of $f$ with respect to~$\mathbf{G}$.
\end{theorem}

In this work, we adhere to situation (A), where the commutator calculus from the previous section applies. Case~(B) is included in order to show that the identities \eqref{dfG} and \eqref{dfG2} are compatible with the Gateaux derivative, which is the universal notion of a~functional derivative. Note that the expression $f(\mathbf{G}+s\mathbf{X})$ cannot be interpreted using \eqref{matf} as~$\mathbf{X}$ might not be symmetric, and therefore case (B) relies on the power series expansion of~$f$. Of course, one could search for further (or more general) scenarios than (A), (B), where \eqref{dfG} still provides a~meaningful notion of the matrix derivative, but this is not our point here.

Any of the identities \eqref{dfG}, \eqref{dfG2} and \eqref{dfG0} can be seen as a~coordinate-free version of the Dalecki\u{\i}--Kre\u{\i}n formula \eqref{KD}; compare also with the original double operator integral form, e.g.,~in \cite[(1.12)]{Birman_2003} or \cite[(6)]{Farforovskaya_1998}. We note that \eqref{dfG2} in particular can be also viewed as a~generalization of the classical identity (cf.~\cite{Wilcox1967})
\begin{equation}\label{wilc0}
	\frac{de^{\mathbf{G}}}{d\mathbf{G}}\mathbf{X}=\int_0^1e^{(1-t)\mathbf{G}}\mathbf{X} e^{t\mathbf{G}}dt,
\end{equation}
where the matrix exponential is replaced by an~arbitrary function~$f$.

Identity \eqref{dfG0} is really just \eqref{KD} after an~application of Definition~\ref{Def}. Note that it aligns with the formal integration in \eqref{dfG2} according to the fundamental theorem of calculus. One only has to keep in mind that \eqref{dfG0} is to be interpreted as an~evaluation of the whole function 
\begin{equation*}
    (x,y)\mapsto\frac{f(x)-f(y)}{x-y} \qquad\text{at}\qquad(\ml_{\mathbf{G}},\mr_{\mathbf{G}}).
\end{equation*}
It transpires below (e.g.~in Section~\ref{Sproof}) that \eqref{dfG0} is very practical for the symbolic calculations of matrix derivatives. In fact, it reduces the problem to finding a~``nice'', well-defined expression for the \emph{scalar} function $\frac{f(x)-f(y)}{x-y}$, which is where various integral representations often come into play. Moreover, identity \eqref{dfG0} actually requires the least regularity from the function~$f$ in order to make sense (e.g.~integrability of the first derivatives of~$f$ in $U\times U$ would suffice). On the other hand, it does not immediately allow for a~power-series interpretation, unlike \eqref{dfG} and \eqref{dfG2}. Note also that one can compose both sides of \eqref{dfG0} with $\ad_{\mathbf{G}}$ and apply Theorem~\ref{Lrep}~(iii), (ix). This leads to
\begin{equation}\label{adf}
	\ad_{f(\mathbf{G})}=\frac{df(\mathbf{G})}{d\mathbf{G}}\ad_{\mathbf{G}},
\end{equation}
which retrieves the well-known chain rule for the matrix commutator, cf.~\cite[Lem.~1]{Powers_1973}, see also \cite[(1.9)]{Birman_2003} for general self-adjoint operators.

\subsection{Formulae for derivatives of special matrix functions}

Now we have all the tools ready to calculate a~symbolic derivative of a~particular function. We remark that the task of differentiating a~general matrix function was already investigated in the 1950s, see \cite{Rinehart_1957} and subsequent developments in \cite{Powers_1973} and references therein. Here, we shall be predominantly interested in the exponential, logarithm, trigonometric, and power functions, as these are fundamental in applications. Unlike in the scalar case, the same matrix derivative can always be written in multiple ways, which can be useful in various situations. The list of formulae in the following theorem includes those that are known in literature (see~\cite[X.4]{Bhatia_1997} in particular); however, we provide alternative (simple) proofs based on the formula \eqref{dfG} or \eqref{dfG0}, see Section~\ref{Sproof}.

\begin{theorem}\label{Tform}
	Let $\mathbf{G}\in\mathbb{R}^{d\times d}_{\rm sym}$ and let $\mathbf{A}\in\mathbb{R}^{d\times d}_{\rm sym}$ be positive definite. Then, for the derivatives of the classical matrix functions, the following identities hold:
	\begin{itemize}
		\item[\rm(E)] \textbf{Exponential}
		\begin{align}
			\frac{de^{\mathbf{G}}}{d\mathbf{G}}&=\int_0^1\ml_{e^{(1-s)\mathbf{G}}}\mr_{e^{s\mathbf{G}}}ds.\label{E0}\\
			&=\ml_{e^{\mathbf{G}}}\frac{1-e^{-2\ad_{\mathbf{G}}}}{2\ad_{\mathbf{G}}}\label{E1}\\
			&=\mr_{e^{\mathbf{G}}}\frac{e^{2\ad_{\mathbf{G}}}-1}{2\ad_{\mathbf{G}}}\label{E2}\\
			&=\ac_{e^{\mathbf{G}}}\frac{\tanh\ad_{\mathbf{G}}}{\ad_{\mathbf{G}}}\label{E3}\\
			&=e^{\ac_{\mathbf{G}}}\frac{\sinh\ad_{\mathbf{G}}}{\ad_{\mathbf{G}}}\label{E4}.
		\end{align}
		
		\item[\rm(L)] \textbf{Logarithm}
		\begin{align}			
			\frac{d\log\mathbf{A}}{d\mathbf{A}}&=\int_0^1\frac{ds}{(1-s)\ml_{\mathbf{A}}+s\mr_{\mathbf{A}}}\label{L5}\\
			&=\int_0^1\ml_{((1-s)\mathbf{I}+s\mathbf{A})^{-1}}\mr_{((1-s)\mathbf{I}+s\mathbf{A})^{-1}}ds\label{L6}\\
			&=\ml_{\mathbf{A}^{-1}}\frac{2\ad_{\log\mathbf{A}}}{1-e^{-2\ad_{\log\mathbf{A}}}}\label{L1}\\
			&=\mr_{\mathbf{A}^{-1}}\frac{2\ad_{\log\mathbf{A}}}{e^{2\ad_{\log\mathbf{A}}}-1}\label{L2}\\
			&=\ac_{\mathbf{A}^{-1}}\frac{2\ad_{\log\mathbf{A}}}{\sinh(2\ad_{\log\mathbf{A}})}\label{L3}\\
			&=e^{-\ac_{\log\mathbf{A}}}\frac{\ad_{\log\mathbf{A}}}{\sinh\ad_{\log\mathbf{A}}}\label{L4}\\
			&=\Big(\frac{de^{\mathbf{G}}}{d\mathbf{G}}\Big)^{-1}_{\mathbf{G}=\log\mathbf{A}}.\label{L0}
		\end{align}
		
		\item[\rm(P)] \textbf{The $r$-th power}
		\begin{align}
			\frac{d\mathbf{A}^r}{d\mathbf{A}}&=r\int_0^1\big((1-s)\ml_{\mathbf{A}}+s\mr_{\mathbf{A}}\big)^{r-1}ds\label{PP0}\\
			&=\sum_{k=0}^{r-1}\ml_{\mathbf{A}}^k\mr_{\mathbf{A}}^{r-1-k}\quad(\text{if}\quad r\in\mathbb{N})\label{PP00}\\
			&=\ac_{\mathbf{A}^{r-1}}\Big(1+\frac{\tanh((r-1)\ad_{\log\mathbf{A}})}{\tanh\ad_{\log\mathbf{A}}}\Big)\label{PP1}\\
			&=e^{(r-1)\ac_{\log\mathbf{A}}}\frac{\sinh(r\ad_{\log\mathbf{A}})}{\sinh\ad_{\log\mathbf{A}}}.\label{PP2}
		\end{align}
		
		\item[\rm(H)] \textbf{Hyperbolic functions}
		\begin{align}
			\frac{d\cosh\mathbf{G}}{d\mathbf{G}}&=\ac_{\sinh\mathbf{G}}\frac{\tanh\ad_{\mathbf{G}}}{\ad_{\mathbf{G}}}\\
			&=\sinh\ac_{\mathbf{G}}\frac{\sinh\ad_{\mathbf{G}}}{\ad_{\mathbf{G}}};\\
			\frac{d\sinh\mathbf{G}}{d\mathbf{G}}&=\ac_{\cosh\mathbf{G}}\frac{\tanh\ad_{\mathbf{G}}}{\ad_{\mathbf{G}}}\\
			&=\cosh\ac_{\mathbf{G}}\frac{\sinh\ad_{\mathbf{G}}}{\ad_{\mathbf{G}}}.
		\end{align}
		
		\item[\rm(G)] \textbf{Trigonometric functions}
		\begin{align}
			\frac{d\cos\mathbf{G}}{d\mathbf{G}}&=-\ac_{\sin\mathbf{G}}\frac{\tan\ad_{\mathbf{G}}}{\ad_{\mathbf{G}}}\\
			&=-\sin\ac_{\mathbf{G}}\frac{\sin\ad_{\mathbf{G}}}{\ad_{\mathbf{G}}};\\
			\frac{d\sin\mathbf{G}}{d\mathbf{G}}&=\ac_{\cos\mathbf{G}}\frac{\tan\ad_{\mathbf{G}}}{\ad_{\mathbf{G}}}\\
			&=\cos\ac_{\mathbf{G}}\frac{\sin\ad_{\mathbf{G}}}{\ad_{\mathbf{G}}}.
		\end{align}
	\end{itemize}
\end{theorem}

In Theorem~\ref{Tform}, expressions like $\frac{\sinh\ad_{\mathbf{G}}}{\ad_{\mathbf{G}}}$ and similar are understood as whole, i.e., as an~application of the function $x\mapsto\frac{\sinh x}x$ (having a~removable discontinuity at zero) to the operator $\ad_{\mathbf{G}}$ according to Definition~\ref{Def}, recall also \eqref{comap}.

Formula \eqref{E1} for the derivative of the matrix exponential appears probably in every book on Lie algebras; see, for instance, \cite[Th.~5.4]{Hall_2015} or \cite[Sect.~1.2., Th.~5]{Rossmann2002}. Its counterpart \eqref{E2}, which is in fact just a~combination of \eqref{E1} with the Hadamard lemma (cf.~Lemma~\ref{Had} below), is much less common, but sometimes appears in proofs of the Baker--Campbell--Hausdorff formula. Symmetrized identities \eqref{E3}, \eqref{E4} can be found in \cite[Sect.~4]{Najfeld_1995} and in \cite[Th.~10.13]{Higham_2008}, however encoded in a~different language, using the Kronecker representation, which eventually relies on the interpretation of power series. The inverted versions \eqref{L1}--\eqref{L4} seem to be not stated in literature this explicitly and for an~arbitrary positive definite matrix $\mathbf{A}$. It also appears that the potential strength of \eqref{L4} in physics and PDE applications has not been accordingly utilized so far, see Section~\ref{Sapp}~below. The identity \eqref{E0} is well known and is the same as \eqref{wilc0} mentioned above. Formula \eqref{L5} is the exact counterpart of \eqref{wilc0}, although it is unlikely to be found in the literature, since it is based on Definition~\ref{Def}. On the other hand, identity \eqref{L6} can be directly interpreted in a~usual way and it is indeed well known (cf.~\cite[(11.10)]{Higham_2008}), since it follows from the integral definition of the matrix logarithm (see \cite[(11.1)]{Higham_2008}). The formula for the derivative of the matrix power \eqref{PP2} can be directly applied to matrix-valued partial differential equations with diffusion; see Section~\ref{SPDE} below. A~similar formula to \eqref{PP0} is derived in \cite{Bhatia_2009} to solve a certain matrix equation, which can also be directly treated using \eqref{dfG0}. Finally, regarding the formulae for the matrix hyperbolic and trigonometric functions, these are obviously direct consequences of the previous identities for the matrix exponential. It seems that the only relatable result in the literature is \cite[Th.~12.4]{Higham_2008}, which uses the Kronecker representation.

We shall demonstrate in Section~\ref{Sapp} below that Theorem~\ref{Tform} is very useful in \emph{symbolic} matrix computations that are then used to obtain theoretical results. On the other hand, the question of optimal \emph{numerical} evaluation of the above derivatives and matrix functions in general is a~very different and highly non-trivial problematic, to which there corresponds an~enormous literature, see the overview \cite{Moler_2003} and references therein just for the computation of the matrix exponential. Let us at least provide the following worked-out example that should clarify how the above formulae can be evaluated concretely.

\begin{example}\label{example}
	Let us examine the identity
	\begin{equation}\label{L3b}
		\frac{d\log\mathbf{A}}{d\mathbf{A}}=\ac_{\mathbf{A}^{-1}}\frac{2\ad_{\log\mathbf{A}}}{\sinh(2\ad_{\log\mathbf{A}})}
	\end{equation}
from Theorem~\ref{Tform}. Note again that the right-hand side is the composition of two fourth-order tensors: part $\ac_{\mathbf{A}^{-1}}$ is expected to be there in the commutative case and the rest depends on the commutator. For simplicity, consider the case $d=2$. 
	
	Although we could apply Theorem~\ref{TCayley} directly, we shall use a little trick which consists of writing
	\begin{equation}\label{sinh}
		\frac{2x}{\sinh(2x)}=1+\theta(x)x,
	\end{equation}
	where
	\begin{equation*}
        \theta(x):=\frac2{\sinh(2x)}-\frac1x,\quad x\neq0.
    \end{equation*}
	The point is that the function $\theta$ is odd and thus, computing $\theta(\ad_{\log\mathbf{A}})$ now becomes very easy according to \eqref{cayley2d}:
	\begin{equation}\label{eq1a}
		\theta(\ad_{\log\mathbf{A}})\mathbf{X}=\frac{\theta(\frac{\log a_1-\log a_2}2)}{\log a_1-\log a_2}\big((\log\mathbf{A})\mathbf{X}-\mathbf{X}(\log\mathbf{A})\big),
	\end{equation}
	where $a_1$, $a_2$ are eigenvalues of $\mathbf{A}$. Noting also that $\frac{\theta(x)}{x}=\frac{\theta(|x|)}{|x|}$ and denoting the invariant
	\begin{equation*}
        \delta:=\frac12\sqrt{\operatorname{tr}^2\log\mathbf{A}-4\det\log\mathbf{A}}=\frac12|\log a_1-\log a_2|,
    \end{equation*}
	identity \eqref{eq1a} becomes just
	\begin{equation}\label{eq2a}
		\theta(\ad_{\log\mathbf{A}})\mathbf{X}=\frac{\theta(\delta)}{\delta}\ad_{\log\mathbf{A}}\mathbf{X}.
	\end{equation}
	Using this, \eqref{sinh} and the commutator calculus rules, the right-hand side of \eqref{L3b} simplifies as
	\begin{equation}
		\frac{d\log\mathbf{A}}{d\mathbf{A}}=\ac_{\mathbf{A}^{-1}}+\frac{\theta(\delta)}{\delta}\ac_{\mathbf{A}^{-1}}\ad_{\log\mathbf{A}}^2.
	\end{equation}
	In conclusion, in two dimensions, the derivative of the matrix logarithm can be evaluated according to
	\begin{align}
		\frac{d\log\mathbf{A}}{d\mathbf{A}}\mathbf{X}&=\lim_{s\to0}\frac{\log(\mathbf{A}+s\mathbf{X})-\log\mathbf{A}}{s}\nonumber\\
		&=\frac12(\mathbf{A}^{-1}\mathbf{X}+\mathbf{X}\mathbf{A}^{-1})\nonumber\\
        &\qquad+\frac{\theta(\delta)}{8\delta}\mathbf{A}^{-1}\big((\log\mathbf{A})^2\mathbf{X}-2(\log\mathbf{A})\mathbf{X}(\log\mathbf{A})+\mathbf{X}(\log\mathbf{A})^2\big)\nonumber\\
		&\qquad+\frac{\theta(\delta)}{8\delta}\big((\log\mathbf{A})^2\mathbf{X}-2(\log\mathbf{A})\mathbf{X}(\log\mathbf{A})+\mathbf{X}(\log\mathbf{A})^2)\mathbf{A}^{-1}\big).\label{res}
	\end{align}
	This could be, of course, further manipulated using, e.g., the Cayley--Hamilton theorem for $\log\mathbf{A}$, however the core idea should already be clear. Note that the resulting formula \eqref{res} indeed only sees the matrix $\mathbf{A}$ and not its spectral decomposition.
\end{example}

Returning to the statement of Theorem~\ref{Tform}, we now explain the apparent dichotomy between \eqref{E3} and \eqref{E4} and so on. This is caused by the unique properties of the exponential map.

\begin{lemma}\label{Had}
	Let $p,q\in\mathbb{R}$, $\mathbf{G}\in\mathbb{R}^{d\times d}_{\rm sym}$ and set $\mathbf{A}=e^{\mathbf{G}}$. Then
	\begin{align}
		\ml_{\mathbf{A}^p}\mr_{\mathbf{A}^q}&=e^{(p+q)\ac_{\mathbf{G}}}e^{(p-q)\ad_{\mathbf{G}}};\label{al1}\\
		\ac_{\mathbf{A}}&=e^{\ac_{\mathbf{G}}}\cosh\ad_{\mathbf{G}};\label{al2}\\
		\ad_{\mathbf{A}}&=e^{\ac_{\mathbf{G}}}\sinh\ad_{\mathbf{G}};\label{al3}\\
		\ad_{\mathbf{A}}&=\ac_{\mathbf{A}}\tanh\ad_{\mathbf{G}};\label{al4}\\
		\ac_{\mathbf{G}}&=\frac12(\log\ml_{\mathbf{A}}+\log\mr_{\mathbf{A}});\label{al5}\\
		\ad_{\mathbf{G}}&=\frac12(\log\ml_{\mathbf{A}}-\log\mr_{\mathbf{A}}).\label{al6}
	\end{align}
\end{lemma}
\begin{proof}
	First we recall the general property of the exponential map (defined by its Taylor series) that
	\begin{equation}\label{expcom}
		e^{\mathsf{X}}e^{\mathsf{Y}}=e^{\mathsf{X}+\mathsf{Y}}\quad\text{whenever the operators}\quad\mathsf{X},\mathsf{Y}\in\mathcal{L}(\mathbb{R}^{d\times d})\quad\text{commute},
	\end{equation}
	see, e.g., \cite[p.~3, (b)]{Rossmann2002}. Hence, applying also \eqref{funcm} gives
	\begin{equation*}
		\ml_{\mathbf{A}^p}\mr_{\mathbf{A}^q}=\ml_{e^{p\mathbf{G}}}\mr_{e^{q\mathbf{G}}}=e^{p\ml_{\mathbf{G}}}e^{q\mr_{\mathbf{G}}}=e^{p\ml_{\mathbf{G}}+q\mr_{\mathbf{G}}}
	\end{equation*}
	and thus \eqref{al1} follows. For \eqref{al2} and \eqref{al3}, we again use \eqref{funcm} to get
	\begin{equation*}
		\begin{matrix}e^{\ac_{\mathbf{G}}}\cosh\ad_{\mathbf{G}}\\e^{\ac_{\mathbf{G}}}\sinh\ad_{\mathbf{G}}\end{matrix}\Big\}=\frac{e^{\ac_{\mathbf{G}}+\ad_{\mathbf{G}}}\pm e^{\ac_{\mathbf{G}}-\ad_{\mathbf{G}}}}2=\frac{e^{\ml_{\mathbf{G}}}\pm e^{\mr_{\mathbf{G}}}}2=\frac{\ml_{\mathbf{A}}\pm\mr_{\mathbf{A}}}2=\big\{\begin{matrix}\ac_{\mathbf{A}}\\\ad_{\mathbf{A}}\end{matrix}.
	\end{equation*}
	Identity \eqref{al4} can be obtained by combining \eqref{al2} and \eqref{al3}:
	\begin{equation*}
		\ad_{\mathbf{A}}=e^{\ac_{\mathbf{G}}}\sinh\ad_{\mathbf{G}}=e^{\ac_{\mathbf{G}}}\cosh\ad_{\mathbf{G}}\tanh\ad_{\mathbf{G}}=\ac_{\mathbf{A}}\tanh\ad_{\mathbf{G}}.
	\end{equation*}
	Finally, the last two relations \eqref{al5} and \eqref{al6} follow immediately by writing $\mathbf{G}=\log\mathbf{A}$ and using \eqref{funcm}.
\end{proof}

The identity \eqref{al1} applies to an~expression of the type $\mathbf{A}^p\mathbf{X}\mathbf{A}^q$ and can be seen as its decomposition in the symmetric multiplication of $\mathbf{X}$ by $\mathbf{A}^{p+q}$ and the correction due to the non-commutativity of $\mathbf{A}$ and $\mathbf{X}$. The special case $p=\frac12$, $q=-\frac12$ of \eqref{al1} recovers Campbell's identity (also called Hadamard's lemma in quantum mechanics) in the form
\begin{equation}\label{hdd}
	\mathbf{A}^{\frac12}\mathbf{X}\mathbf{A}^{-\frac12}=e^{\ad_{\log\mathbf{A}}}\mathbf{X},
\end{equation}
see, e.g.~\cite[Prop.~3.35]{Hall_2015}. Another distinctive case $p=\frac12$, $q=\frac12$ of \eqref{al1} together with \eqref{al2} shows that
\begin{equation}\label{hac}
	\mathbf{A}^{\frac12}\mathbf{X}\mathbf{A}^{\frac12}=e^{\ac_{\log\mathbf{A}}}\mathbf{X}=\frac1{2\cosh\ad_{\log\mathbf{A}}}(\mathbf{A}\mathbf{X}+\mathbf{X}\mathbf{A}),
\end{equation}
which is well-defined thanks to Theorem~\ref{Lrep}~(iv) and provides another handy tool for symbolic manipulations. The identity \eqref{al4} represents probably the most direct relation between the commutator and the anti-commutator.

\subsection{Further identities for the derivative of the matrix power and logarithm}

In numerous physics applications, one is interested in an~application of a~matrix derivative to a~certain linear combination of the left and right matrix multiplications. In particular, one would like to have matrix analogies of the standard calculus identities $\frac{da^r}{da}(ax)=ra^rx$ and $\frac{d\log a}{da}(ax)=x$ when $ax$ is replaced by $\mathbf{A}\mathbf{X}+\mathbf{X}\mathbf{A}$, for example (note that the case $\mathbf{A}\mathbf{X}-\mathbf{X}\mathbf{A}$ is already dealt with by \eqref{adf}, but we keep it in the corollaries below for comparison). Thanks to Theorem~\ref{Tform}, which expresses the Fr\'echet derivatives explicitly, the commutator calculus properties from Theorem~\ref{Lrep} and Lemma~\ref{Had}, this is now rather easy. As an~immediate consequence of \eqref{PP2} and \eqref{al1}, we obtain the following.

\begin{corollary}\label{corpow}
	For every $\mathbf{A}\in\mathbb{R}^{d\times d}_{\rm sym}$ positive definite, all $\mathbf{X}\in\mathbb{R}^{d\times d}$ and any $p,q,r\in\mathbb{R}$, there holds
	\begin{align}
		\frac{d\mathbf{A}^r}{d\mathbf{A}}(\mathbf{A}^p\mathbf{X}\mathbf{A}^q-\mathbf{A}^q\mathbf{X}\mathbf{A}^p)=2e^{(p+q+r-1)\ac_{\log\mathbf{A}}}\frac{\sinh(r\ad_{\log\mathbf{A}})\sinh((p-q)\ad_{\log\mathbf{A}})}{\sinh\ad_{\log\mathbf{A}}}\mathbf{X};\label{W1}\\
		\frac{d\mathbf{A}^r}{d\mathbf{A}}(\mathbf{A}^p\mathbf{X}\mathbf{A}^q+\mathbf{A}^q\mathbf{X}\mathbf{A}^p)=2e^{(p+q+r-1)\ac_{\log\mathbf{A}}}\frac{\sinh(r\ad_{\log\mathbf{A}})\cosh((p-q)\ad_{\log\mathbf{A}})}{\sinh\ad_{\log\mathbf{A}}}\mathbf{X}.\label{W2}
	\end{align}
\end{corollary}

The previous result is trivial when $r=0$. Instead, the appropriate limiting version as $r\to0$ is captured by the next result. As it concerns the derivative of the matrix logarithm, it is of great importance in applications within continuum mechanics, see the next section. 

\begin{corollary}\label{Cor}
	For every $\mathbf{A}\in\mathbb{R}^{d\times d}_{\rm sym}$ positive definite, all $\mathbf{X}\in\mathbb{R}^{d\times d}$ and any $p,q\in\mathbb{R}$, there holds
	\begin{align}
		\frac{d\log\mathbf{A}}{d\mathbf{A}}(\mathbf{A}^p\mathbf{X}\mathbf{A}^q-\mathbf{A}^q\mathbf{X}\mathbf{A}^p)&=2e^{(p+q-1)\ac_{\log\mathbf{A}}}\frac{\sinh((p-q)\ad_{\log\mathbf{A}})}{\sinh\ad_{\log\mathbf{A}}}\ad_{\log\mathbf{A}}\mathbf{X};\label{O0}\\
		\frac{d\log\mathbf{A}}{d\mathbf{A}}(\mathbf{A}^p\mathbf{X}\mathbf{A}^q+\mathbf{A}^q\mathbf{X}\mathbf{A}^p)&=2e^{(p+q-1)\ac_{\log\mathbf{A}}}\frac{\cosh((p-q)\ad_{\log\mathbf{A}})}{\sinh\ad_{\log\mathbf{A}}}\ad_{\log\mathbf{A}}\mathbf{X}.\label{O1}
	\end{align}
	In particular, the case $p=1$, $q=0$ gives
	\begin{align}
		\frac{d\log\mathbf{A}}{d\mathbf{A}}(\mathbf{A}\mathbf{X}-\mathbf{X}\mathbf{A})&=2\ad_{\log\mathbf{A}}\mathbf{X};\label{P1}\\
		\frac{d\log\mathbf{A}}{d\mathbf{A}}(\mathbf{A}\mathbf{X}+\mathbf{X}\mathbf{A})&=\frac{2\ad_{\log\mathbf{A}}}{\tanh\ad_{\log\mathbf{A}}}\mathbf{X},\label{P2}
	\end{align}
	while the case $p=q=\frac12$ yields
	\begin{equation}\label{P3}
		\frac{d\log\mathbf{A}}{d\mathbf{A}}\mathbf{A}^{\frac12}\mathbf{X}\mathbf{A}^{\frac12}=\frac{\ad_{\log\mathbf{A}}}{\sinh\ad_{\log\mathbf{A}}}\mathbf{X}.
	\end{equation}
	
	Relations \eqref{P2} and \eqref{P3} can be further refined as
	\begin{align}
		\frac{d\log\mathbf{A}}{d\mathbf{A}}(\mathbf{A}\mathbf{X}+\mathbf{X}\mathbf{A})&=2\mathbf{X}+2\mu(\ad_{\log\mathbf{A}})\ad_{\log\mathbf{A}}^2\mathbf{X};\label{spec1}\\
		\frac{d\log\mathbf{A}}{d\mathbf{A}}\mathbf{A}^{\frac12}\mathbf{X}\mathbf{A}^{\frac12}&=\mathbf{X}-\nu(\ad_{\log\mathbf{A}})\ad_{\log\mathbf{A}}^2\mathbf{X},\label{spec2}
	\end{align}
	where
	\begin{equation*}
		\mu(x):=\frac1{x\tanh{x}}-\frac1{x^2}\quad\text{and}\quad\nu(x):=\frac1{x^2}-\frac1{x\sinh{x}}
	\end{equation*}
	are positive, continuous functions on $\mathbb{R}$ if $\mu(0):=\frac13$ and $\nu(0):=\frac16$. Moreover, there holds
	\begin{align}
		\frac{d\log\mathbf{A}}{d\mathbf{A}}(\mathbf{A}\mathbf{X}+\mathbf{X}\mathbf{A})=2\mathbf{X}\quad&\text{if and only if}\quad\mathbf{A}\mathbf{X}=\mathbf{X}\mathbf{A};\label{eq1}\\
		\frac{d\log\mathbf{A}}{d\mathbf{A}}\mathbf{A}^{\frac12}\mathbf{X}\mathbf{A}^{\frac12}=\mathbf{X}\quad&\text{if and only if}\quad\mathbf{A}\mathbf{X}=\mathbf{X}\mathbf{A}.\label{eq2}
	\end{align}
\end{corollary}
\begin{proof}
	From \eqref{L4} and \eqref{al1}, we infer that
	\begin{align*}
		\frac{d\log\mathbf{A}}{d\mathbf{A}}(\ml_{\mathbf{A}^p}\mr_{\mathbf{A}^q}\mp\ml_{\mathbf{A}^q}\mr_{\mathbf{A}^p})&=e^{(p+q-1)\ac_{\log\mathbf{A}}}\frac{\ad_{\log\mathbf{A}}}{\sinh\ad_{\log\mathbf{A}}}(e^{(p-q)\ad_{\log\mathbf{A}}}\mp e^{(q-p)\ad_{\log\mathbf{A}}}),
	\end{align*}
	which is \eqref{O0} and \eqref{O1} up to a rearrangement. Identities \eqref{P1}, \eqref{P2}, and \eqref{P3} are really just particular cases of \eqref{O0} and \eqref{O1}. 
	
	In order to show \eqref{spec1}, note that
	\begin{equation}\label{exp1}
		\frac{x}{\tanh x}=1+\mu(x)x^2,
	\end{equation}
	and that $\mu$ is positive since
	\begin{equation*}
		x\tanh{x}=|x||\tanh{x}|<|x|^2=x^2
	\end{equation*}
	for every $x\neq0$.	Analogously, for \eqref{spec2}, we have
	\begin{equation}\label{exp2}
		\frac{x}{\sinh{x}}=1-\nu(x)x^2
	\end{equation}
	and the positivity of $\nu$ follows from
	\begin{equation*}
		x\sinh{x}=|x||\sinh{x}|>|x|^2=x^2.
	\end{equation*}
	With the help of the calculus rules in Theorem~\ref{Lrep}, we can now apply \eqref{exp1} and \eqref{exp2} in \eqref{P2} and \eqref{P3}, leading to \eqref{spec1} and \eqref{spec2}, respectively. This also readily proves the reverse implications in \eqref{eq1} and \eqref{eq2}. Conversely, assume that $\mu(\ad_{\log\mathbf{A}})\ad_{\log\mathbf{A}}^2\mathbf{X}={\boldsymbol0}$. By the positivity of $\mu$ and Theorem~\ref{Lrep}~(iv), this is equivalent to $\ad_{\log\mathbf{A}}^2\mathbf{X}={\boldsymbol0}$. However, due to the symmetry of $\mathbf{A}$ and using Theorem~\ref{Lrep}~(vi), this actually implies that $\ad_{\log\mathbf{A}}\mathbf{X}={\boldsymbol0}$ since
	\begin{equation*}
		0=\mathbf{X}\cdot\ad_{\log\mathbf{A}}^2\mathbf{X}=|\ad_{\log\mathbf{A}}\mathbf{X}|^2,
	\end{equation*}
	Then, for instance due to \eqref{al4}, we have also that
	\begin{equation*}
		\ad_{\mathbf{A}}\mathbf{X}=\ac_{\mathbf{A}}\tanh\ad_{\log\mathbf{A}}\mathbf{X}=\ac_{\mathbf{A}}\frac{\tanh\ad_{\log\mathbf{A}}}{\ad_{\log\mathbf{A}}}\ad_{\log\mathbf{A}}\mathbf{X}={\boldsymbol0},
	\end{equation*}
	proving \eqref{eq1}. The proof of the other equivalence \eqref{eq2} is completely analogous.
\end{proof}

The idea of rewriting \eqref{P2} as \eqref{spec1} can be used in many of the identities of Theorem~\ref{Tform} and is particularly useful for quantifying the dependence of the derivative on the commutator of $\mathbf{G}$ (or $\mathbf{A}$) and $\mathbf{X}$. We see that the commutator calculus reduces the problem to just finding limits and determining the sign of the \emph{real} functions. This provides a~concise approach to many of the problems studied in \cite{Neff2024}, cf.~Prop.~A.31~therein in particular.

\subsection{Possible extension to non-symmetric matrices}

The assumption of symmetry of the matrix $\mathbf{G}$ leads to the fact that $\mathbf{G}$ itself and also the operators derived from $\ml_{\mathbf{G}}$, $\mr_{\mathbf{G}}$ are self-adjoint, with real eigenvalues. Thus, this case represents a~natural extension of the classical \emph{real-variable} calculus. This is an~optimal setting for the applications in continuum mechanics (see Sect.~\ref{Sold} below), where the matrices can represent physical stresses or strains, which, at least within standard theories, are indeed elements of $\mathbb{R}^{3\times 3}_{\rm sym}$.

However, note that the identities stated in Theorem~\ref{TCayley}, Theorem~\ref{Tform}, Lemma~\ref{Had} or Corollaries~\ref{corpow}, \ref{Cor} do not rely on the spectral decomposition of~$\mathbf{G}$ in any explicit way. The only exceptions are the relations \eqref{eq1} and \eqref{eq2}, which exploit the symmetry of $\mathbf{A}$ only to reduce the characterizing condition $\ad_{\mathbf{A}}^2\mathbf{X}=\mathbf{0}$. Although Definition~\ref{Def} is built using Schur's spectral decomposition, we can infer from Theorem~\ref{Lrep}~(viii) and, in particular from Theorem~\ref{TCayley}, that it is rather only an~auxiliary tool for defining the symbol $f(\ml_{\mathbf{G}},\mr_{\mathbf{G}})$. Based on that, we make a~following conjecture left to prove (or disprove) elsewhere.

\begin{conjecture}\label{Conj}
	With a~suitable generalization of Definition~\ref{Def}, all identities of Theorem~\ref{Tform}, Lemma~\ref{Had}, and Corollaries~\ref{corpow},~\ref{Cor} (except \eqref{eq1}, \eqref{eq2}) continue to hold for every $\mathbf{G}\in\mathbb{C}^{d\times d}$ and for any $\mathbf{A}\in\mathbb{C}^{d\times d}$ such that $\log\mathbf{A}$ exists.
\end{conjecture}

For the definition of a~logarithm of a~general matrix, we can refer to \cite[Th.~1]{Wouk_1965}. Then, there appear to be numerous approaches towards defining $f(\ad_{\mathbf{G}})$ for $\mathbf{G}\in\mathbb{C}^{d\times d}$. First, of course, one can apply the general theory of holomorphic calculus to the operator $\ad_{\mathbf{G}}$ (and similar). However, this requires that the applied function is holomorphic in an~open subset of $\mathbb{C}$ that contains the spectrum of the operator, even in the case where the spectrum is actually real. The second possibility is to replace the Schur decomposition of $\mathbf{G}$ by the singular value decomposition. This obviously places the least requirements on the applied functions; on the other hand, one has to consider an~appropriate notion of a~generalized matrix function and its derivative, leading to more involved formulae than those in \eqref{had} and \eqref{KD}, see \cite{Israel_1973} and especially \cite{Noferini_2017} for details. Another approach is to use the spectral measure, that also offers a~way to extend the theory to general self-adjoint operators. For this, see the original Russian literature, cf.~\cite{Birman_2003}, \cite{Krein1956}, \cite{Farforovskaya_1998} and references therein. Further alternative approaches to defining and evaluating matrix functions and their derivatives (e.g.~using polynomial interpolation), can be found in the works by Higham \& Al-Mohy (cf.~\cite{Mohy2024}, \cite{Mohy_2010}). As we are here primarily concerned with applications of the commutator calculus to symmetric matrices, we avoid any of the above technicalities.

\section{Applications of the theory}\label{Sapp}

In this section, we demonstrate the utility of our results in several applications, first within general mathematical analysis and then for the theory of viscoelastic fluids in particular.

\subsection{Monotonicity of matrix functions}

Formula \eqref{dfG0} together with Theorem~\ref{Lrep} can be used to easily derive various theoretical properties of matrix functions. One~particularly useful property in analysis is the \emph{monotonicity}, which, for matrix functions, can be postulated as the inequality
\begin{equation*}
    \big(f(\mathbf{G})-f(\mathbf{H})\big)\cdot(\mathbf{G}-\mathbf{H})\geq0
\end{equation*}
for all $\mathbf{G},\mathbf{H}\in\mathbb{R}^{d\times d}_{\rm sym}$ whose spectra lie in some interval $I\subset\mathbb{R}$ where $f$ is defined. 

Using polynomial interpolation methods, it can be shown that if $f\in\mathcal{C}^1(I;\mathbb{R})$, then the function
\begin{equation*}
    g(s):=f(\mathbf{H}+s(\mathbf{G}-\mathbf{H}))\cdot(\mathbf{G}-\mathbf{H})
\end{equation*}
inherits the property $g\in\mathcal{C}^1(I;\mathbb{R})$, see~\cite[Th.~6.6.14]{Horn_1991}. Hence, one can apply the mean value theorem to $g$ and obtain this way $\mathbf{G}_{\xi}:=\mathbf{H}+\xi(\mathbf{G}-\mathbf{H})$, $\xi\in(0,1)$, such that
\begin{align*}
	\big(f(\mathbf{G})-f(\mathbf{H})\big)\cdot(\mathbf{G}-\mathbf{H})&=\frac{df(\mathbf{G}_{\xi})}{d\mathbf{G}_{\xi}}(\mathbf{G}-\mathbf{H})\cdot(\mathbf{G}-\mathbf{H})\\
	&=\frac{f(\ml_{\mathbf{G}_{\xi}})-f(\mr_{\mathbf{G}_{\xi}})}{\ml_{\mathbf{G}_{\xi}}-\mr_{\mathbf{G}_{\xi}}}(\mathbf{G}-\mathbf{H})\cdot(\mathbf{G}-\mathbf{H})
\end{align*}
where we used \eqref{dfG0} in the last step. Thus, in view of the multiplicative, invertible, and symmetry properties from Theorem~\ref{Lrep}~(iii), (iv) and (iv), respectively, the monotonicity problem is reduced to a simple investigation of the sign of the function $(x,y)\mapsto\frac{f(x)-f(y)}{x-y}$ in $I\times I$, that is, the monotonicity of $f$ in the usual sense.

\begin{corollary}
	Let $f\in\mathcal{C}^1(I;\mathbb{R})$ be a nondecreasing function in an interval $I\subset\mathbb{R}$. Then, for every $\mathbf{G},\mathbf{H}\in\mathbb{R}^{d\times d}_{\rm sym}$ whose eigenvalues are contained in $I$, there exists $\xi\in(0,1)$ such that
	\begin{equation}\label{mono}
        \big(f(\mathbf{G})-f(\mathbf{H})\big)\cdot(\mathbf{G}-\mathbf{H})=\Bigg|\sqrt{\frac{f(\ml_{\mathbf{G}_{\xi}})-f(\mr_{\mathbf{G}_{\xi}})}{\ml_{\mathbf{G}_{\xi}}-\mr_{\mathbf{G}_{\xi}}}}(\mathbf{G}-\mathbf{H})\Bigg|^2,
    \end{equation}
	where $\mathbf{G}_{\xi}=(1-\xi)\mathbf{H}+\xi\mathbf{G}$.
\end{corollary}

This agrees with the results of \cite[Prop.~4.8]{Martin_2015} obtained by different methods.

\subsection{Log-convexity of the Sobolev norm for matrix-valued functions}\label{SPDE}

For a~smooth and positive scalar function $u:\mathbb{R}^d\to(0,\infty)$ it is an~elementary fact that
\begin{equation}\label{pwid}
	\frac1{r}\nabla u^{r}\cdot\nabla u=\frac{4}{(r+1)^2}|\nabla u^{\frac{r+1}2}|^2
\end{equation}
for any $r\neq0,-1$. This identity is often crucial for the analysis of partial differential equations with some kind of a~diffusion term. Therein, the left-hand side of \eqref{pwid} is typically controlled, while the right-hand side of \eqref{pwid} allows for an~application of the Sobolev embedding inequality, often leading to optimal integrability estimates of solutions. For matrix-valued PDEs, when $u$ is replaced by a~positive definite matrix function $\mathbf{B}$, it has been proved (cf.~\cite{BathoryOdh}) that \eqref{pwid} can only hold as an~inequality
\begin{equation}\label{pwidB}
	\frac1{r}\nabla\mathbf{B}^{r}\cdot\nabla\mathbf{B}\geq\frac{4}{(r+1)^2}|\nabla\mathbf{B}^{\frac{r+1}2}|^2.
\end{equation}
However, the direction of this inequality is somewhat remarkable; it implies that the a~priori control of $\frac1{r}\nabla\mathbf{B}^{r}\cdot\nabla\mathbf{B}$ conveys more information in the non-commutative case than in the commutative case. Indeed, now using the commutator calculus, we can prove the explicit formula
\begin{equation}\label{expl}
	\frac1r\nabla\mathbf{B}^r\cdot\nabla\mathbf{B}=\frac{4}{(r+1)^2}\Big(|\nabla\mathbf{B}^{\frac{r+1}2}|^2+|\omega(\ad_{\log\mathbf{B}})\ad_{\log\mathbf{B}}\nabla\mathbf{B}^{\frac{r+1}2}|^2\Big),
\end{equation}
where
\begin{equation*}
	\omega(x):=\sqrt{\frac{(r+1)^2\sinh(rx)\sinh{x}}{4rx^2\sinh^2(\frac{r+1}2x)}-\frac1{x^2}}
\end{equation*}
for $x\neq0$ and $\omega(0):=\frac{|r-1|}{2\sqrt3}$ defines a~continuous \emph{positive} function on $\mathbb{R}$. This shows that the inequality \eqref{pwidB} is strict in general and provides a~certain bound also for the commutator of $\mathbf{B}$ and $\nabla\mathbf{B}$. In order to see this even more explicitly, we can apply Theorem~\ref{TCayley} to the odd function $x\omega(x)$. In the two-dimensional case, for instance, this transforms \eqref{expl} into
\begin{equation}\label{expl2}
	\frac1r\nabla\mathbf{B}^r\cdot\nabla\mathbf{B}=\frac{4}{(r+1)^2}\Big(|\nabla\mathbf{B}^{\frac{r+1}2}|^2+\omega(\delta)^2|\ad_{\log\mathbf{B}}\nabla\mathbf{B}^{\frac{r+1}2}|^2\Big)
\end{equation}
for any $\mathbf{B}\in\mathbb{R}^{2\times2}_{\rm sym}$ is positive definite, and where
\begin{equation*}
	\delta:=\sqrt{\frac14\operatorname{tr}^2\log\mathbf{B}-\det\log\mathbf{B}}.
\end{equation*}

In order to see \eqref{expl}, we use \eqref{PP2} and get
\begin{align*}
	\nabla\mathbf{A}^{2(1-\alpha)}\cdot\nabla\mathbf{A}^{2\alpha}&=\frac{\sinh(2(1-\alpha)\ad_{\log\mathbf{A}})\sinh(2\alpha\ad_{\log\mathbf{A}})}{\sinh^2\ad_{\log\mathbf{A}}}\nabla\mathbf{A}\cdot\nabla\mathbf{A}
\end{align*}
for a~smooth positive definite tensor field $\mathbf{A}:\mathbb{R}^d\to\mathbb{R}^{d\times d}_{\rm sym}$. The choice
\begin{equation*}
	\mathbf{A}=\mathbf{B}^{\frac1{2\alpha}}\quad\text{and}\quad\alpha=\frac1{r+1}
\end{equation*}
leads to
\begin{equation*}
	\nabla\mathbf{B}^r\cdot\nabla\mathbf{B}=\frac{\sinh(r\ad_{\log\mathbf{B}})\sinh\ad_{\log\mathbf{B}}}{\sinh^2(\frac{r+1}2\ad_{\log\mathbf{B}})}\nabla\mathbf{B}^{\frac{r+1}2}\cdot\nabla\mathbf{B}^{\frac{r+1}2}.
\end{equation*}
Observing now that the fraction has a~limit $4r/(r+1)^2$ at zero, it is natural to rewrite it as \eqref{expl} and investigate the sign of $\omega$. In the last term of \eqref{expl} one can further apply Theorem~\ref{Tform} and Lemma~\ref{Had} to rewrite it using a~desired commutator, for instance $\ad_{\log\mathbf{B}}\nabla\log\mathbf{B}$.

More abstractly, instead of \eqref{expl}, one can show, for any nondecreasing absolutely continuous function $f:I\to\mathbb{R}$, where the interval $I$ contains the spectrum of $\mathbf{B}$, that
\begin{equation}
	\nabla f(\mathbf{B})\cdot\nabla\mathbf{B}=\Bigg|\sqrt{\frac{f(\ml_{\mathbf{B}})-f(\mr_{\mathbf{B}})}{\ml_{\mathbf{B}}-\mr_{\mathbf{B}}}}\nabla\mathbf{B}\Bigg|^2.
\end{equation}

\subsection{Inequalities involving matrix functions}

The results of \cite{BathoryOdh} are based on the inequality
\begin{equation}\label{logconv}
	\mathsf{P}(r)\mathbf{X}\cdot\mathsf{P}(-r)\mathbf{X}\geq |\mathsf{P}(0)\mathbf{X}|^2,
\end{equation}
which holds for all $r\in\mathbb{R}$ and $\mathbf{X}\in\mathbb{R}^{d\times d}$, where
\begin{equation}\label{Pdef}
	\mathsf{P}(r)\mathbf{X}:=\int_0^1\mathbf{A}^{(1+r)s}\mathbf{X}\mathbf{A}^{-(1+r)s}ds
\end{equation}
for some positive definite $\mathbf{A}\in\mathbb{R}^{d\times d}_{\rm sym}$.	Using the commutator calculus, we now see that \eqref{Pdef} is tantamount to
\begin{equation*}
	\mathsf{P}(r)=\frac{e^{(1+r)\ell}-1}{(1+r)\ell}\quad\text{with}\quad\ell:=2\ad_{\log\mathbf{A}},
\end{equation*}
and we can actually calculate the error $\mathsf{P}(r)\mathbf{X}\cdot\mathsf{P}(-r)\mathbf{X}-|\mathsf{P}(0)\mathbf{X}|^2$ explicitly, providing an~alternative proof of \eqref{logconv}. In fact, we have
\begin{align*}
	\mathsf{P}(r)\mathbf{X}&\cdot\mathsf{P}(-r)\mathbf{X}-|\mathsf{P}(0)\mathbf{X}|^2\\
	&=\Big(\frac1{1-r^2}\frac{e^{2\ell}-e^{(1-r)\ell}-e^{(1+r)\ell}+1}{\ell^2}-\frac{e^{2\ell}-2e^{\ell}+1}{\ell^2}\Big)\mathbf{X}\cdot\mathbf{X}\\
	&=\frac{2e^{\ell}}{\ell^2}\Big(\frac1{1-r^2}(\cosh\ell-\cosh(r\ell))-(\cosh\ell-1)\Big)\mathbf{X}\cdot\mathbf{X}
\end{align*}
where the expression in parentheses is positive if understood as a~real function. This can be best seen by applying the Taylor series of $\cosh$ to get
\begin{align*}
	\frac1{1-r^2}(\cosh\ell-\cosh(r\ell))-(\cosh\ell-1)&=\sum_{n=1}^{\infty}\frac1{(2n)!}\frac{r^2-r^{2n}}{1-r^2}\ell^{2n}\\
	&=\sum_{n=1}^{\infty}\frac1{(2n+2)!}\sum_{k=1}^{n}r^{2k}\ell^{2n+2}
\end{align*}
and thus
\begin{align*}
	\mathsf{P}(r)\mathbf{X}\cdot\mathsf{P}(-r)\mathbf{X}-|\mathsf{P}(0)\mathbf{X}|^2&=\sum_{n=1}^{\infty}\frac{\sum_{k=1}^{n}r^{2k}}{(2n+2)!}|\ell^{n}e^{\frac{\ell}2}\mathbf{X}|^2\\
	&=\sum_{n=1}^{\infty}\frac{2^{2n}\sum_{k=1}^{n}r^{2k}}{(2n+2)!}|\mathbf{A}(\ad_{\log\mathbf{A}}^n\mathbf{X})\mathbf{A}^{-1}|^2\geq0.
\end{align*}

\subsection{Logarithmic reformulation of the classical viscoelastic fluid models}\label{Sold}

The original motivation for this work is the problem of rewriting the famous Oldroyd-B equation \cite{Oldroyd1950},
\begin{equation}\label{OB}
	\partial_t\mathbf{B}+{\boldsymbol v}\cdot\nabla\mathbf{B}+\frac1{\tau}(\mathbf{B}-\mathbf{I})=(\nabla{\boldsymbol v})\mathbf{B}+\mathbf{B}(\nabla{\boldsymbol v})^{\mathsf{T}},\qquad\tau>0,
\end{equation}
for an~unknown positive definite tensor field $\mathbf{B}$ and a~given velocity field ${\boldsymbol v}$ in terms of the tensor $\mathbf{H}=\frac12\log\mathbf{B}$. This is a~natural idea, since if $\mathbf{B}$ was a~scalar, one would like to divide the equation by $\mathbf{B}$ in a~way that $\mathbf{B}$ disappears from the right-hand side. With the help of the above calculus and especially Corollary~\ref{Cor}, this is now indeed possible if certain terms are hidden in the definition of corotational derivative of $\mathbf{H}$, which in turn leads to the celebrated logarithmic corotational objective derivative; see~\cite{Xiao1997} and further references in~\cite{lograte2025}. The correct way to divide the equation is, of course, to apply the operator $\frac{d\mathbf{H}}{d\mathbf{B}}=\frac12\frac{d\log\mathbf{B}}{d\mathbf{B}}$ to both sides, leading to
\begin{equation}\label{Heq}
	\partial_t\mathbf{H}+{\boldsymbol v}\cdot\nabla\mathbf{H}+\frac1{2\tau}(\mathbf{I}-e^{-2\mathbf{H}})=\frac{d\mathbf{H}}{d\mathbf{B}}(\mathbf{W}\mathbf{B}-\mathbf{B}\mathbf{W})+\frac{d\mathbf{H}}{d\mathbf{B}}(\mathbf{D}\mathbf{B}+\mathbf{B}\mathbf{D}),
\end{equation}
where
\begin{equation*}
	\mathbf{W}:=\frac{\nabla{\boldsymbol v}-(\nabla{\boldsymbol v})^{\mathsf{T}}}2\quad\text{and}\quad\mathbf{D}:=\frac{\nabla{\boldsymbol v}+(\nabla{\boldsymbol v})^{\mathsf{T}}}2.
\end{equation*}
The first term on the right-hand side of \eqref{Heq} can be readily written according to the chain rule \eqref{adf} as
\begin{equation}\label{ft}
	\frac{d\mathbf{H}}{d\mathbf{B}}(\mathbf{W}\mathbf{B}-\mathbf{B}\mathbf{W})=-\frac{d\log\mathbf{B}}{d\mathbf{B}}\ad_{\mathbf{B}}\mathbf{W}=-\ad_{\log\mathbf{B}}\mathbf{W}=-2\ad_{\mathbf{H}}\mathbf{W}.
\end{equation}
Most importantly, however, the last term in \eqref{Heq} can be dealt with by \eqref{P2} as follows
\begin{equation}\label{st}
	\frac{d\mathbf{H}}{d\mathbf{B}}(\mathbf{D}\mathbf{B}+\mathbf{B}\mathbf{D})=\ad_{\log\mathbf{B}}\coth\ad_{\log\mathbf{B}}\mathbf{D}=2\ad_{\mathbf{H}}\coth(2\ad_{\mathbf{H}})\mathbf{D}
\end{equation}
Using \eqref{ft} and \eqref{st} in \eqref{Heq}, expanding $2x\coth(2x)=1+2x(\coth(2x)-1/2x)$ and rearranging, we finally arrive at
\begin{equation}\label{Heq2}
	\partial_t\mathbf{H}+{\boldsymbol v}\cdot\nabla\mathbf{H}+2\ad_{\mathbf{H}}{\boldsymbol\Omega}^{\log}+\frac1{2\tau}(\mathbf{I}-e^{-2\mathbf{H}})=\mathbf{D},
\end{equation}
where
\begin{equation}\label{logspin}
	{\boldsymbol\Omega}^{\log}:=\mathbf{W}-\mathcal{L}(2\ad_{\mathbf{H}})\mathbf{D}
\end{equation}
and
\begin{equation}
	\mathcal{L}(x):=\Big\{\begin{matrix}\coth x-\frac1x&x\neq0,\\0&x=0\end{matrix}
\end{equation}
is the Langevin function. Note that $\mathcal{L}$ is continuous and odd, hence the tensor ${\boldsymbol\Omega}^{\log}$ is well defined by \eqref{logspin} (appealing to Definition~\ref{Def}) and antisymmetric due to Theorem~\ref{Lrep}~(v).

It turns out that the tensor ${\boldsymbol\Omega}^{\log}$ obtained in this way coincides with the logarithmic spin tensor of Xiao~et.~al., see \cite{Xiao1997}, where it has been used primarily in models of generalized elasticity; therefore, we use the same notation. A~detailed study of this connection, using the classical power series approach, can be found in our related work \cite{lograte2025}. In reality, it is easy to check that all \emph{material} spin tensors $\boldsymbol{\Omega}$ (as defined in \cite{Xiao_1998}) are precisely of the form
\begin{equation}\label{spingen}
	\boldsymbol{\Omega}=\mathbf{W}-f(\ad_{\mathbf{H}})\mathbf{D},\qquad f\text{ continuous and odd.}
\end{equation}
This observation is in line with the findings of \cite{Meng2022} and \cite{Norris2008} where different basis-free formulae are found for various spin tensors. The general formula \eqref{spingen} refines these results by identifying the fourth-order tensor applied to $\mathbf{D}$ as $f(\ad_{\mathbf{H}})$, which can be further evaluated explicitly in the selected dimension using Theorem~\ref{TCayley}. For the exact correspondence between some of the most important physical spin tensors and the choice of the function $f$ in \eqref{spingen}, see the examples in Sec.~4 of \cite{Xiao_1998} (and apply the transformation $f(x)=-\tilde{h}(e^{2x})$ therein).

Let us return to the Oldroyd viscoelastic model \eqref{OB} and its equivalent logarithmic version \eqref{Heq2}. The first three terms of \eqref{Heq2} give rise to the logarithmic corotational derivative (logarithmic rate)
\begin{equation}\label{lograte}
	\accentset{\boldsymbol\circ}{\mathbf{H}}^{\log}:=\partial_t\mathbf{H}+{\boldsymbol v}\cdot\nabla\mathbf{H}+2\ad_{\mathbf{H}}{\boldsymbol\Omega}^{\log},
\end{equation}
turning \eqref{Heq2} into
\begin{equation}\label{Hlog}
	\accentset{\boldsymbol\circ}{\mathbf{H}}^{\log}+\frac1{2\tau}(\mathbf{I}-e^{-2\mathbf{H}})=\mathbf{D},
\end{equation}
which is the true logarithmic reformulation of the original model
\begin{equation}\label{old}
	\accentset{\boldsymbol\nabla}{\mathbf{B}}+\frac1{\tau}(\mathbf{B}-\mathbf{I})={\boldsymbol0}
\end{equation}
due to Oldroyd, where
\begin{equation}\label{upc}
	\accentset{\boldsymbol\nabla}{\mathbf{B}}:=\partial_t\mathbf{B}+{\boldsymbol v}\cdot\nabla\mathbf{B}-(\nabla{\boldsymbol v})\mathbf{B}-\mathbf{B}(\nabla{\boldsymbol v})^{\mathsf{T}}
\end{equation}
is the upper-convected derivative.

The linearization of the second term of \eqref{Hlog} leads to the simpler model
\begin{equation}\label{Hloglin}
	\accentset{\boldsymbol\circ}{\mathbf{H}}^{\log}+\frac1{\tau}\mathbf{H}=\mathbf{D},
\end{equation}
which, after applying the backward transformation $\frac{d\mathbf{B}}{d\mathbf{H}}=\frac{de^{2\mathbf{H}}}{d\mathbf{H}}$, can be seen to be equivalent to the model 
\begin{equation}\label{giu}
	\accentset{\boldsymbol\nabla}{\mathbf{B}}+\frac1{\tau}\mathbf{B}\log\mathbf{B}={\boldsymbol0}
\end{equation}
proposed in \cite[(8)]{Alrashdi2024} using different arguments. Since
\begin{equation*}
    \mathbf{B}\log\mathbf{B}=\frac12(\mathbf{B}-\mathbf{I})+\frac12(\mathbf{B}^2-\mathbf{B})+O(|\mathbf{B}-\mathbf{I}|^3)\quad\text{as}\quad|\mathbf{B}-\mathbf{I}|\to0,
\end{equation*}
one can see \eqref{giu} (and thus also \eqref{Hloglin}) as a certain interpolation between the Oldroyd and Giesekus viscoelastic models. One can also compare \eqref{Hlog} or \eqref{Hloglin} with~\cite[(C1)]{Giusteri_2024}, where it is now clear that one would have to take $\mathsf{S}={\boldsymbol0}$ and $\mathsf{\Omega}=\mathbf{W}-\coth(2\ad_{\log\mathbf{B}})$, that is, \emph{singular} Eulerian twirl tensor. Moreover, there are indications (cf.~\cite{Masmoudi2011}), that if \eqref{Hloglin} is coupled to the Navier-Stokes equations in a~natural way (as in \cite{Giusteri_2024}), then the resulting system of partial differential equations admits possibly even a~three-dimensional global weak solution for large initial data. We shall investigate this in detail elsewhere.

\subsection{Connecting the logarithmic and upper-convected rates}

Computations of the previous subsection in fact reveal an~explicit general relation between the upper-convected and logarithmic rates.

\begin{corollary}
	Let $\mathbf{A}$ and $\mathbf{G}$ be symmetric tensor fields such that $\mathbf{A}=e^{\mathbf{G}}$ and let the rates $\accentset{\boldsymbol\circ}{\mathbf{G}}^{\log}$ and $\accentset{\boldsymbol\nabla}{\mathbf{A}}$ be defined as in \eqref{lograte} and \eqref{upc}, respectively, with respect to a~given vector field ${\boldsymbol v}$. Then
	\begin{equation}\label{r1}
		\accentset{\boldsymbol\circ}{\mathbf{G}}^{\log}=2\mathbf{D}+\frac{d\log\mathbf{A}}{d\mathbf{A}}\accentset{\boldsymbol\nabla}{\mathbf{A}}
	\end{equation}
	and, equivalently
	\begin{equation}\label{r2}
		\accentset{\boldsymbol\nabla}{\mathbf{A}}=\frac{de^{\mathbf{G}}}{d\mathbf{G}}(\accentset{\boldsymbol\circ}{\mathbf{G}}^{\log}-2\mathbf{D}).
	\end{equation}
\end{corollary}

The relation \eqref{r1} can be viewed from two interesting angles. Firstly, the fundamental equation of solid mechanics $\accentset{\boldsymbol\circ}{\mathbf{H}}^{\log}=\mathbf{D}$ for Hencky strain $\mathbf{H}$ is just a~special case of \eqref{r1} if $\mathbf{G}=2\mathbf{H}$ and $\mathbf{A}=\mathbf{B}$ since the left Cauchy--Green tensor satisfies $\accentset{\boldsymbol\nabla}{\mathbf{B}}={\boldsymbol0}$ by definition. In particular, the relation \eqref{r1} \emph{directly} shows that $\accentset{\boldsymbol\nabla}{\mathbf{B}}={\boldsymbol0}$ is equivalent to $\accentset{\boldsymbol\circ}{\mathbf{H}}^{\log}=\mathbf{D}$.

Secondly, the relation \eqref{r1} provides an~alternative way to define the logarithmic rate $\accentset{\boldsymbol\circ}{\mathbf{G}}^{\log}$ even outside the realm of solid mechanics, where it traditionally relies on the definition of the left Cauchy--Green tensor~$\mathbf{B}$ and its spectral decomposition. Unlike that, since the upper-convected derivative $\accentset{\boldsymbol\nabla}{\mathbf{A}}$ can be directly evaluated from \eqref{upc} for any tensor $\mathbf{A}$ irrespective of its physical meaning, identity \eqref{r1}, written more succinctly as
\begin{equation}\label{r11}
	\accentset{\boldsymbol\circ}{\mathbf{G}}^{\log}=2\mathbf{D}+\Big(\frac{d\log\mathbf{A}}{d\mathbf{A}}\Big)_{\mathbf{A}=e^{\mathbf{G}}}\accentset{\boldsymbol\nabla}{(e^{\mathbf{G}})},
\end{equation}
requires only the symbolic computation of the derivative of the matrix logarithm. In~view of \eqref{ft} and \eqref{st}, the identity \eqref{r11} is consistent with \eqref{lograte} and justifies the choice \eqref{logspin} of logarithmic spin dependent on $\mathbf{H}$ regardless of the constraints of solid mechanics. Otherwise, hypothetically, one could replace $\mathbf{H}$ in \eqref{logspin} with some fixed tensor $\mathbf{H}_0$, where the theory of solids corresponds to the choice $\mathbf{H}_0=\frac12\log\mathbf{B}$.

Analogous relations to \eqref{r1} and \eqref{r2} can, of course, be deduced for other types of objective derivative used in various physics applications; see \cite[(1.13), (1.16)]{Neff_2025} for further examples. We also refer to \cite{Aubram}, \cite{Fiala_2019}, \cite{Neff_2025}, \cite{Vejvoda_2025} and references therein for the state-of-the-art developments and applications of objective rates and logarithmic strain in continuum mechanics.

\subsection{Comparing dissipation potentials}

When equation \eqref{old} is supplied with the diffusion term $-\Delta\mathbf{B}$, it corresponds to the additional term in the dissipation of the form
\begin{equation*}
	|\mathbf{B}^{-\frac12}\nabla\mathbf{B}\mathbf{B}^{-\frac12}|^2=\Delta\mathbf{B}\cdot\mathbf{B}^{-1}-\Delta\log\det\mathbf{B},
\end{equation*}
see, e.g., \cite{Bathory_2020} for details. This term is roughly of the same order as $|\nabla\log\mathbf{B}|^2$, that is, $|\nabla\mathbf{H}|^2$, but the precise relationship between the two has not been clear. Now, simply using \eqref{al1}, \eqref{E4} and the chain rule gives
\begin{equation*}
	\mathbf{B}^{-\frac12}\nabla\mathbf{B}\mathbf{B}^{-\frac12}=e^{-\ac_{2\mathbf{H}}}\frac{de^{2\mathbf{H}}}{d\mathbf{H}}\nabla\mathbf{H}=\frac{\sinh(2\ad_{\mathbf{H}})}{\ad_{\mathbf{H}}}\nabla\mathbf{H}.
\end{equation*}
Thus, taking advantage of the symmetry from Theorem~\ref{Lrep}~(vi) and expanding the function $\frac{\sinh^2(2x)}{x^2}$ into its Taylor series leads to
\begin{align}
	|\mathbf{B}^{-\frac12}\nabla\mathbf{B}\mathbf{B}^{-\frac12}|^2&=|\nabla\mathbf{H}|^2+\sum_{n=1}^{\infty}\frac{2}{(2n+2)!}|\ad_{\mathbf{H}}^n\nabla\mathbf{H}|^2\nonumber\\
	&=|\nabla\mathbf{H}|^2+\frac1{12}|\ad_{\mathbf{H}}\nabla\mathbf{H}|^2+\frac1{360}|\ad_{\mathbf{H}}^2\nabla\mathbf{H}|^2+\ldots,
\end{align}
indicating that the simplest diffusion term in the formulation with $\mathbf{B}$ has generally a~stronger dissipative effect than its counterpart in the logarithmic formulation using $\mathbf{H}$.

Similar ideas can be used to compare the expressions of type $|\mathbf{B}^p\nabla\mathbf{B}^r\mathbf{B}^q|^2$ for different values of exponents $p,q,r\in\mathbb{R}$, among each other.

\section{Proofs}\label{Sproof}

\subsection{Functions applied to the commutator}

We remind the reader that by $[m_{ij}]$, we denote a~$d\times d$ matrix with elements $m_{ij}$, $i,j=1,\ldots,d$.

\begin{proof}[Proof of Theorem~\ref{Lrep}]
	(ii) If $\mathbf{G}$ and $\mathbf{H}$ commute, then there exists a~Schur diagonalization of both $\mathbf{G}$ and $\mathbf{H}$ with the common matrix $\mathbf{Q}$, see \cite{HornHong1985}. Let $\{g_i\}_{i=1}^d$ and $\{h_i\}_{i=1}^d$ be the corresponding eigenvalues. Then, we just expand Definition~\ref{Def}, apply $\mathbf{Q}^{\mathsf{T}}\mathbf{Q}=\mathbf{I}$ and use the commutativity and associativity of the $\odot$ product:
	\begin{align*}
		f_1(\ml_{\mathbf{G}},\mr_{\mathbf{G}})(f_2(\ml_{\mathbf{H}},\mr_{\mathbf{H}})\mathbf{X})&=\mathbf{Q}\big([f_1(g_i,g_j)]\odot[f_2(h_i,h_j)]\odot(\mathbf{Q}^{\mathsf{T}}\mathbf{X}\mathbf{Q})\big)\mathbf{Q}^{\mathsf{T}}\\
		&=\mathbf{Q}\big([f_2(h_i,h_j)]\odot [f_1(g_i,g_j)]\odot(\mathbf{Q}^{\mathsf{T}}\mathbf{X}\mathbf{Q})\big)\mathbf{Q}^{\mathsf{T}}\\
		&=f_2(\ml_{\mathbf{H}},\mr_{\mathbf{H}})(f_1(\ml_{\mathbf{G}},\mr_{\mathbf{G}})\mathbf{X})
	\end{align*}
	
	(iii) These properties are merely consequences of \eqref{had} combined with
	\begin{equation*}
		[(f_1+f_2)(g_i,g_j)]=[f_1(g_i,g_j)]+[f_2(g_i,g_j)]
	\end{equation*}
	and
	\begin{equation*}
		[(f_1f_2)(g_i,g_j)]=[f_1(g_i,g_j)]\odot [f_2(g_i,g_j)],
	\end{equation*}
	respectively.
	
	(i) First of all, let us verify the consistency of \eqref{had} in the simplest of cases. If $f(x,y)=1$, then indeed
	\begin{equation}\label{id1}
		\mathbf{Q}\big([1]\odot(\mathbf{Q}^{\mathsf{T}}\mathbf{X}\mathbf{Q})\big)\mathbf{Q}^{\mathsf{T}}=\mathbf{Q}(\mathbf{Q}^{\mathsf{T}}\mathbf{X}\mathbf{Q})\mathbf{Q}^{\mathsf{T}}=\mathbf{X}
	\end{equation}
	defines an~identity element in $\mathcal{L}(\mathbb{R}^{d\times d})$. Next, if $f(x,y)=x$, then
	\begin{equation}\label{idx}
		\mathbf{Q}\big([g_i]\odot(\mathbf{Q}^{\mathsf{T}}\mathbf{X}\mathbf{Q})\big)\mathbf{Q}^{\mathsf{T}}=\mathbf{Q}(\mathbf{g}\mathbf{Q}^{\mathsf{T}}\mathbf{X}\mathbf{Q})\mathbf{Q}^{\mathsf{T}}=\mathbf{G}\mathbf{X}=\ml_{\mathbf{G}}\mathbf{X},
	\end{equation}
	which represents the matrix multiplication from the left by $\mathbf{G}$. Analogously, if $f(x,y)=y$, then
	\begin{equation}\label{idy}
		\mathbf{Q}\big([g_j]\odot(\mathbf{Q}^{\mathsf{T}}\mathbf{X}\mathbf{Q})\big)\mathbf{Q}^{\mathsf{T}}=\mathbf{Q}(\mathbf{Q}^{\mathsf{T}}\mathbf{X}\mathbf{Q}\mathbf{g})\mathbf{Q}^{\mathsf{T}}=\mathbf{X}\mathbf{G}=\mr_{\mathbf{G}}\mathbf{X},
	\end{equation}
	retrieves the matrix multiplication by $\mathbf{G}$ from the right. From \eqref{id1}--\eqref{idy}, (ii) and (iii), we infer that
	\begin{equation*}
		p(\ml_{\mathbf{G}},\mr_{\mathbf{G}})=\Big(\sum_{m=0}^M\sum_{n=0}^Np_{mn}x^my^n\Big)(\ml_{\mathbf{G}},\mr_{\mathbf{G}})=\sum_{m=0}^M\sum_{n=0}^Np_{mn}\ml_{\mathbf{G}}^m\mr_{\mathbf{G}}^n,
	\end{equation*}
	which indeed coincides with the usual interpretation of the symbol $p(\ml_{\mathbf{G}},\mr_{\mathbf{G}})$.
	
	(iv) Given any $\mathbf{Y}\in\mathbb{R}^{d\times d}$, it is evident from (iii) that
	\begin{equation*}
		\mathbf{X}=\mathbf{Q}\Big(\Big[\frac1{f(g_i,g_j)}\Big]\odot(\mathbf{Q}^{\mathsf{T}}\mathbf{Y}\mathbf{Q})\Big)\mathbf{Q}^{\mathsf{T}}=(1/f)(\ml_{\mathbf{G}},\mr_{\mathbf{G}})\mathbf{Y}
	\end{equation*}
	is the unique element $\mathbf{X}$ satisfying $f(\ml_{\mathbf{G}},\mr_{\mathbf{G}})\mathbf{X}=\mathbf{Y}$.
	
	(v) This is an~easy consequence of the fact that
	\begin{equation*}
		[f(g_i,g_j)]^{\mathsf{T}}=[f(g_j,g_i)].
	\end{equation*}
	
	(vi) This property follows simply from the computation
	\begin{align*}
		f(\ml_{\mathbf{G}},\mr_{\mathbf{G}})\mathbf{X}\cdot\mathbf{Y}&=\big([f(g_i,g_j)]\odot(\mathbf{Q}^{\mathsf{T}}\mathbf{X}\mathbf{Q})\big)\cdot\mathbf{Q}^{\mathsf{T}}\mathbf{Y}\mathbf{Q}\\
		&=\mathbf{Q}^{\mathsf{T}}\mathbf{X}\mathbf{Q}\cdot\big([f(g_i,g_j)]\odot(\mathbf{Q}^{\mathsf{T}}\mathbf{Y}\mathbf{Q})\big)\\
		&=\mathbf{X}\cdot f(\ml_{\mathbf{G}},\mr_{\mathbf{G}})\mathbf{Y}.
	\end{align*}
	
	(vii) Property \eqref{normf} is a~direct consequence of \eqref{had}, $|\mathbf{Q}|=1$ and the inequality
	\begin{equation*}
		|\mathbf{A}\odot\mathbf{B}|\leq|\mathbf{A}||\mathbf{B}|,\quad\mathbf{A},\mathbf{B}\in\mathbb{R}^{d\times d},
	\end{equation*}
	for which we refer to \cite{HornJohnson1987}.
	
	(viii) Since the determinant of the Vandermonde matrix satisfies
	\begin{equation*}
		\det\mathbf{V}_{\mathbf{G}}=\prod_{1\leq k<l\leq d}(g_l-g_k),
	\end{equation*}
	the linear problem \eqref{lp} admits a~unique solution if all the eigenvalues of $\mathbf{G}$ are distinct. In case of eigenvalue multiplicities, duplicate entries can be eliminated in $\mathbf{V}_{\mathbf{G}}$ and $[f(g_i,g_j)]$, leading to an~underdetermined problem for $\mathbf{J}_{\mathbf{G}}^f$ that can be solved using generalized matrix inverse methods.
	
	In order to prove \eqref{repreq}, we again exploit that $\mathbf{g}=\mathbf{Q}^{\mathsf{T}}\mathbf{G}\mathbf{Q}$ is diagonal and write
	\begin{align*}
		\sum_{p=1}^d\sum_{r=1}^d(\mathbf{J}_{\mathbf{G}}^f)_{pr}\mathbf{G}^{p-1}\mathbf{X}\mathbf{G}^{r-1}&=\mathbf{Q}\Big(\sum_{p=1}^d\sum_{r=1}^d(\mathbf{J}_{\mathbf{G}}^f)_{pr}\mathbf{g}^{p-1}\mathbf{Q}^{\mathsf{T}}\mathbf{X}\mathbf{Q}\mathbf{g}^{r-1}\Big)\mathbf{Q}^{\mathsf{T}}\\
		&=\mathbf{Q}\Big[\sum_{p=1}^d\sum_{r=1}^dg_i^{p-1}(\mathbf{J}_{\mathbf{G}}^f)_{pr}g_j^{r-1}(\mathbf{Q}^{\mathsf{T}}\mathbf{X}\mathbf{Q})_{ij}\Big]\mathbf{Q}^{\mathsf{T}}\\
		&=\mathbf{Q}\big((\mathbf{V}_{\mathbf{G}}\mathbf{J}_{\mathbf{G}}^f\mathbf{V}_{\mathbf{G}}^{\mathsf{T}})\odot(\mathbf{Q}^{\mathsf{T}}\mathbf{X}\mathbf{Q})\big)\mathbf{Q}^{\mathsf{T}}\\
		&=\mathbf{Q}\big([f(g_i,g_j)]\odot(\mathbf{Q}^{\mathsf{T}}\mathbf{X}\mathbf{Q})\big)\mathbf{Q}^{\mathsf{T}}\\
		&=f(\ml_{\mathbf{G}},\mr_{\mathbf{G}})\mathbf{X}.
	\end{align*}
\end{proof}

\begin{proof}[Proof of Theorem~\ref{TCayley}]
	(i) If $d=1$, then $\ad_{\mathbf{G}}=\mathsf{0}$, leading to \eqref{cayley1d}. 
	
	Else, let us denote the right-hand sides of \eqref{cayley2d} and \eqref{cayley3d} by $\mathsf{RHS}_{\mathbf{G}}\mathbf{X}$. It is enough to show that \eqref{cayley2d} and \eqref{cayley3d} hold in the diagonal case, i.e., that
	\begin{equation*}
		f(\ad_{\mathbf{g}})\mathbf{Y}=\mathsf{RHS}_{\mathbf{g}}\mathbf{Y}\quad\text{for all}\quad\mathbf{Y}\in\mathbb{R}^{d\times d}
	\end{equation*}
	if $\mathbf{g}=\operatorname{diag}(g_i)_{i=1}^d$. Indeed, the general case where $\mathbf{G}=\mathbf{Q}\mathbf{g}\mathbf{Q}^{\mathsf{T}}$ then follows by observing that
	\begin{align*}
		f(\ad_{\mathbf{G}})\mathbf{X}&=\mathbf{Q}\Big(\Big[f\Big(\frac{g_i-g_j}2\Big)\Big]\odot(\mathbf{Q}^{\mathsf{T}}\mathbf{X}\mathbf{Q})\Big)\mathbf{Q}^{\mathsf{T}}\\
		&=\mathbf{Q}\big(f(\ad_{\mathbf{g}})(\mathbf{Q}^{\mathsf{T}}\mathbf{X}\mathbf{Q})\big)\mathbf{Q}^{\mathsf{T}}\\
		&=\mathbf{Q}\big(\mathsf{RHS}_{\mathbf{g}}(\mathbf{Q}^{\mathsf{T}}\mathbf{X}\mathbf{Q})\big)\mathbf{Q}^{\mathsf{T}}\\
		&=\mathsf{RHS}_{\mathbf{G}}\mathbf{X}.
	\end{align*}
	Hence, for the rest of the proof, we shall assume $\mathbf{G}=\mathbf{g}$.
	
	(ii) Let us define auxiliary numbers
	\begin{equation*}
		g_{mn}^{ij}:=(g_i-g_m)(g_j-g_n)+(g_i-g_n)(g_j-g_m),\quad i,j,m,n=1,\ldots,d.
	\end{equation*}
	Recall that for $\mathbf{G}$ diagonal, for any $m$ and $n$, there holds
	\begin{align*}
		(\mathbf{G}^m\mathbf{X}\mathbf{G}^n)_{ij}=g_i^mg_j^n\mathbf{X}_{ij},\quad i,j=1,\ldots,d.
	\end{align*}
	Therefore, if $d=2$, we have
	\begin{align}
		&\big(-2g_1g_2\mathbf{X}+(g_1+g_2)(\mathbf{G}\mathbf{X}+\mathbf{X}\mathbf{G})-2\mathbf{G}\mathbf{X}\mathbf{G}\big)_{ij}\nonumber\\
		&\quad=(-2g_1g_2+(g_1+g_2)(g_i+g_j)-2g_ig_j)\mathbf{X}_{ij}\nonumber\\
		&\quad=(-2g_1g_2+g_1g_i+g_1g_j+g_2g_i+g_2g_j-2g_ig_j)\mathbf{X}_{ij}\nonumber\\
		&\quad=-g^{ij}_{12}\mathbf{X}_{ij}\nonumber\\
		&\quad=(g_1-g_2)^2\left[\begin{matrix}
			0&1\\1&0
		\end{matrix}\right]_{ij}\mathbf{X}_{ij}\label{pr2d}
	\end{align}
	and
	\begin{equation}\label{pr2d2}
		(\mathbf{G}\mathbf{X}-\mathbf{X}\mathbf{G})_{ij}=(g_i-g_j)\mathbf{X}_{ij}=(g_1-g_2)\left[\begin{matrix}0&1\\-1&0\end{matrix}\right]_{ij}\mathbf{X}_{ij}
	\end{equation}
	for each $i,j\in\{1,2\}$. Using the last two relations, we find
	\begin{align*}
		(\mathsf{RHS}_{\mathbf{G}}\mathbf{X})_{ij}&=\left[\begin{matrix}
			f_0&\!\!\!\!\!\!\!\!\!\!\!\!f_0+f_{\rm odd}(\frac{g_1-g_2}2)+f_{\rm even}(\frac{g_1-g_2}2)\\
			f_0-f_{\rm odd}(\frac{g_1-g_2}2)+f_{\rm even}(\frac{g_1-g_2}2)&\!\!\!\!\!\!f_0
		\end{matrix}\right]_{ij}\mathbf{X}_{ij}\\
		&=\left[\begin{matrix}
			f_0&f(\frac{g_1-g_2}2)\\
			f(\frac{g_2-g_1}2)&f_0
		\end{matrix}\right]_{ij}\mathbf{X}_{ij}\\
		&=\big(f(\ad_{\mathbf{G}})\mathbf{X}\big)_{ij}.
	\end{align*}	
	
	(iii) In the three-dimensional case, for any $i,j\in\{1,2,3\}$, we further define the vectors
	\begin{equation*}
		\mathbf{w}^{ij}:=\big((g_i-g_1)(g_j-g_1)g^{ij}_{23},(g_i-g_2)(g_j-g_2)g^{ij}_{31},(g_i-g_3)(g_j-g_3)g^{ij}_{12}\big)
	\end{equation*}
	and we compute that
	\begin{align*}
		\mathbf{w}_1^{ij}&=(g_i-g_1)(g_j-g_1)((g_i-g_2)(g_j-g_3)+(g_i-g_3)(g_j-g_2))\\
		&=(g_1^2-g_1(g_i+g_j)+g_ig_j)(2g_2g_3-(g_2+g_3)(g_i+g_j)+2g_ig_j)\\
		&=2g_1^2g_2g_3-(g_1^2(g_2+g_3)+2g_1g_2g_3)(g_i+g_j)+2(g_1^2+g_2g_3)g_ig_j\\
		&\qquad+g_1(g_2+g_3)(g_i+g_j)^2-(2g_1+g_2+g_3)g_ig_j(g_i+g_j)+2g_i^2g_j^2\\
		&=2J_3g_1-(J_2g_1+J_3)(g_i+g_j)+2(g_1^2+J_2)g_ig_j\\
		&\qquad+(g_1J_1-g_1^2)(g_i^2+g_j^2)-(g_1+J_1)g_ig_j(g_i+g_j)+2g_i^2g_j^2.
	\end{align*}
	For the other indices, one can proceed completely analogously, leading to the identity
	\begin{align}
		\mathbf{w}_k^{ij}&=2J_3g_k-(J_2g_k+J_3)(g_i+g_j)+2(g_k^2+J_2)g_ig_j+(g_kJ_1-g_k^2)(g_i^2+g_j^2)\nonumber\\
        &\qquad-(g_k+J_1)g_ig_j(g_i+g_j)+2g_i^2g_j^2\label{Gk}
	\end{align}
	for any $i,j,k\in\{1,2,3\}$. Now we are ready to evaluate the terms on the right-hand side of \eqref{cayley3d}. For the part depending on $f_{\rm odd}$, we have
	\begin{align*}
		\big(-K_2&(\mathbf{G}\mathbf{X}-\mathbf{X}\mathbf{G})+K_1(\mathbf{G}^2\mathbf{X}-\mathbf{X}\mathbf{G}^2)-K_0(\mathbf{G}^2\mathbf{X}\mathbf{G}-\mathbf{G}\mathbf{X}\mathbf{G}^2)\big)_{ij}\\
		&=(g_i-g_j)\big(-K_2+(g_i+g_j)K_1-g_ig_jK_0\big)\\
		&=\frac{(g_i-g_j)\left(\!\!\!\!\!\!\!\!\!\!\!\!\!\!\begin{matrix}(g_1-g_j)(g_i-g_1)f_{\rm odd}(\frac{g_2-g_3}2)\\\qquad+(g_2-g_j)(g_i-g_2)f_{\rm odd}(\frac{g_3-g_1}2)\\\qquad\qquad+(g_3-g_j)(g_i-g_3)f_{\rm odd}(\frac{g_1-g_2}2)\end{matrix}\right)}{(g_1-g_2)(g_2-g_3)(g_3-g_1)}\mathbf{X}_{ij}\\
		&=\left[\begin{matrix}
			0&f_{\rm odd}(\frac{g_1-g_2}2)&-f_{\rm odd}(\frac{g_3-g_1}2)\\
			-f_{\rm odd}(\frac{g_1-g_2}2)&0&f_{\rm odd}(\frac{g_2-g_3}2)\\
			f_{\rm odd}(\frac{g_3-g_1}2)&-f_{\rm odd}(\frac{g_2-g_3}2)&0
		\end{matrix}\right]_{ij}\mathbf{X}_{ij},
	\end{align*}
	while for the other part, we apply \eqref{Gk} to get
	\begin{align*}
		&\big(2J_3L_1\,\mathbf{X}-(J_2L_1+J_3L_0)(\mathbf{G}\mathbf{X}+\mathbf{X}\mathbf{G})\\
		&\qquad+2(L_2+J_2L_0)\,\mathbf{G}\mathbf{X}\mathbf{G}+(J_1L_1-L_2)(\mathbf{G}^2\mathbf{X}+\mathbf{X}\mathbf{G}^2)\\
		&\qquad\qquad-(L_1+J_1L_0)(\mathbf{G}^2\mathbf{X}\mathbf{G}+\mathbf{G}\mathbf{X}\mathbf{G}^2)+2L_0\,\mathbf{G}^2\mathbf{X}\mathbf{G}^2\big)_{ij}\\
		&\qquad=\big(2J_3L_1-(J_2L_1+J_3L_0)(g_i+g_j)+2(L_2+J_2L_0)g_ig_j\\
		&\qquad\qquad+(J_1L_1-L_2)(g_i^2+g_j^2)-(L_1+J_1L_0)g_ig_j(g_i+g_j)+2L_0g_i^2g_j^2\big)\mathbf{X}_{ij}\\
		&\qquad=\frac{\cfrac{f_{\rm even}(\frac{g_1-g_2}2)}{g_1-g_2}\mathbf{w}_3^{ij}+\cfrac{f_{\rm even}(\frac{g_2-g_3}2)}{g_2-g_3}\mathbf{w}_1^{ij}+\cfrac{f_{\rm even}(\frac{g_3-g_1}2)}{g_3-g_1}\mathbf{w}_2^{ij}}{(g_1-g_2)(g_2-g_3)(g_3-g_1)}\mathbf{X}_{ij}\\
		&\qquad=\left[\begin{matrix}
			0&f_{\rm even}(\frac{g_1-g_2}2)&f_{\rm even}(\frac{g_3-g_1}2)\\
			f_{\rm even}(\frac{g_1-g_2}2)&0&f_{\rm even}(\frac{g_2-g_3}2)\\
			f_{\rm even}(\frac{g_3-g_1}2)&f_{\rm even}(\frac{g_2-g_3}2)&0
		\end{matrix}\right]_{ij}\mathbf{X}_{ij}
	\end{align*}
	for all $i,j\in\{1,2,3\}$. Hence, by summing the last two identities and adding also the part $f_0\mathbf{X}$, we arrive at \eqref{cayley3d}.
	
	Regarding the eigenvalue multiplicities, the case of a~single distinct eigenvalue always trivially leads to \eqref{cayley1d}. Next, suppose that $d=3$ and $g_1\neq g_2=g_3$. We proceed just as in \eqref{pr2d} and \eqref{pr2d2} with the difference that $\mathbf{G}$ and $\mathbf{X}$ are now of the size $3\times 3$. We also note that for the additional elements, we have
	\begin{equation*}
		g_{12}^{13}=(g_1-g_2)(g_3-g_1)=-(g_1-g_2)^2,\qquad g_{12}^{23}=0,\qquad g_{12}^{33}=0.
	\end{equation*}
	From this, we obtain
	\begin{equation*}
		\big(-2g_1g_2\mathbf{X}+(g_1+g_2)(\mathbf{G}\mathbf{X}+\mathbf{X}\mathbf{G})-2\mathbf{G}\mathbf{X}\mathbf{G}\big)_{ij}=-g_{12}^{ij}\mathbf{X}_{ij}=(g_1-g_2)^2\left[\begin{matrix}
			0&1&1\\1&0&0\\1&0&0
		\end{matrix}\right]_{ij}\mathbf{X}_{ij}
	\end{equation*}
	and
	\begin{equation*}
		(\mathbf{G}\mathbf{X}-\mathbf{X}\mathbf{G})_{ij}=(g_i-g_j)\mathbf{X}_{ij}=(g_1-g_2)\left[\begin{matrix}0&1&1\\-1&0&0\\-1&0&0\end{matrix}\right]_{ij}\mathbf{X}_{ij},
	\end{equation*}
	for each $i,j\in\{1,2,3\}$, which again leads to the required result. The other two cases $g_2\neq g_3=g_1$ and $g_3\neq g_1=g_2$ are analogous.
\end{proof}

\subsection{Symbolic differentiation of matrix functions}

\begin{proof}[Proof of Theorem~\ref{Tcons}]
	(A) Since $f\in\mathcal{C}^1(I;\mathbb{R})$, by the fundamental theorem of calculus, there holds
	\begin{equation*}
		\int_0^1f'\big((1-s)g_i+sg_j\big)ds=\Big\{\begin{matrix}\frac{f(g_i)-f(g_j)}{g_i-g_j}&g_i\neq g_j\\f'(g_i)&\text{else}\end{matrix}\Big\}=\lim_{(x,y)\to(g_i,g_j)}\frac{f(x)-f(y)}{x-y}.
	\end{equation*}
	Thus, following Definition~\ref{Def} for $f'((1-s)x+sy)$ and then for $\frac{f(x)-f(y)}{x-y}$, we indeed have	
	\begin{align*}
		\frac{df(\mathbf{G})}{d\mathbf{G}}&=\int_0^1f'\big((1-s)\ml_{\mathbf{G}}+\,s\mr_{\mathbf{G}}\big)ds\,\mathbf{X}\\
		&=\mathbf{Q}\Big(\Big[\int_0^1f'\big((1-s)g_i+sg_j\big)ds\Big]\odot(\mathbf{Q}^{\mathsf{T}}\mathbf{X}\mathbf{Q})\Big)\mathbf{Q}^{\mathsf{T}}\\
		&=\mathbf{Q}\Big(\Big[\lim_{(x,y)\to(g_i,g_j)}\frac{f(x)-f(y)}{x-y}\Big]\odot(\mathbf{Q}^{\mathsf{T}}\mathbf{X}\mathbf{Q})\Big)\mathbf{Q}^{\mathsf{T}}\\
		&=\frac{f(\ml_{\mathbf{G}})-f(\mr_{\mathbf{G}})}{\ml_{\mathbf{G}}-\mr_{\mathbf{G}}}.
	\end{align*}
	
	(B) Using the well-known identity for the beta function (\cite[Sect.~6.2]{Owen_1965})
	\begin{equation*}
		\int_0^1(1-s)^ks^{n-1-k}ds=\frac{k!(n-1-k)!}{n!}=\frac1{n\binom{n-1}k},\quad 0\leq k<n,
	\end{equation*}
	and the fact that the operators $\ml_{\mathbf{G}}$ and $\mr_{\mathbf{G}}$ commute for any $\mathbf{G}\in\mathbb{C}^{d\times d}$,
	we see that if $f(x)=x^n$, then
	\begin{align}
		\frac{df(\mathbf{G})}{d\mathbf{G}}&=\int_0^1n\big((1-s)\ml_{\mathbf{G}}+\,s\mr_{\mathbf{G}}\big)^{n-1}ds\nonumber\\
		&=\sum_{k=0}^{n-1}n\binom{n-1}k\int_0^1(1-s)^ks^{n-1-k}ds\ml_{\mathbf{G}}^k\mr_{\mathbf{G}}^{n-1-k}\nonumber\\
		&=\sum_{k=0}^{n-1}\ml_{\mathbf{G}}^k\mr_{\mathbf{G}}^{n-1-k}.\label{wh}
	\end{align}
	This indeed agrees with
	\begin{align*}
		\frac{d}{ds}(\mathbf{G}+s\mathbf{X})^n_{s=0}&=\frac{d}{ds}\Big(\mathbf{G}^n+s\sum_{k=0}^{n-1}\mathbf{G}^k\mathbf{X}\mathbf{G}^{n-1-k}+O(s^2)\Big)_{s=0}=\sum_{k=0}^{n-1}\ml_{\mathbf{G}}^k\mr_{\mathbf{G}}^{n-1-k}\mathbf{X}.
	\end{align*}
	Hence, the case $f(x)=\sum_{n=0}^{\infty}f_nx^n$, $x\in\mathbb{R}$, follows from the previous one by linearity.
\end{proof}

\begin{proof}[Proof of Theorem~\ref{Tform}]
	(E) Property \eqref{E0} is the well-known formula \eqref{wilc0} and it coincides with our definition of derivative in \eqref{dfG2}. Indeed, recalling \eqref{expcom} and \eqref{funcm}, we can write
	\begin{equation*}
		\frac{de^{\mathbf{G}}}{d\mathbf{G}}=\int_0^1e^{(1-t)\ml_{\mathbf{G}}+t\mr_{\mathbf{G}}}dt=\int_0^1e^{(1-t)\ml_{\mathbf{G}}}e^{t\mr_{\mathbf{G}}}dt=\int_0^1\ml_{e^{(1-t)\mathbf{G}}}\mr_{e^{t\mathbf{G}}}dt.
	\end{equation*}
	In order to show \eqref{E2}, for instance, Theorem~\ref{Tcons}, the identity
	\begin{equation*}
		\frac{e^x-e^y}{x-y}=e^y\frac{e^{x-y}-1}{x-y},\quad x\neq y\in\mathbb{R},
	\end{equation*}
	and Theorem~\ref{Lrep}~(iii), (ix) allow one to write
	\begin{align*}
		\frac{de^{\mathbf{G}}}{d\mathbf{G}}&=\frac{e^{\ml_{\mathbf{G}}}-e^{\mr_{\mathbf{G}}}}{\ml_{\mathbf{G}}-\mr_{\mathbf{G}}}=e^{\mr_{\mathbf{G}}}\frac{e^{\ml_{\mathbf{G}}-\mr_{\mathbf{G}}}-1}{\ml_{\mathbf{G}}-\mr_{\mathbf{G}}}=\mr_{e^{\mathbf{G}}}\eta(2\ad_{\mathbf{G}}).
	\end{align*}
	The other variants \eqref{E1} or \eqref{E4} can obviously be obtained by a~slight modification of the above argument. For the sake of completeness, let us also deduce them from \eqref{E2} using \eqref{al1}:
	\begin{align*}
		\frac{de^{\mathbf{G}}}{d\mathbf{G}}=\ml_{e^{\mathbf{G}}}\ml_{e^{-\mathbf{G}}}\mr_{e^{\mathbf{G}}}\eta(2\ad_{\mathbf{G}})=\ml_{e^{\mathbf{G}}}e^{-2\ad_{\mathbf{G}}}\frac{e^{2\ad_{\mathbf{G}}}-1}{2\ad_{\mathbf{G}}}=\ml_{e^{\mathbf{G}}}\eta(-2\ad_{\mathbf{G}})
	\end{align*}
	and
	\begin{equation*}
		\frac{de^{\mathbf{G}}}{d\mathbf{G}}=\ml_{e^{\frac{\mathbf{G}}2}}\mr_{e^{\frac{\mathbf{G}}2}}\ml_{e^{-\frac{\mathbf{G}}2}}\mr_{e^{\frac{\mathbf{G}}2}}\eta(2\ad_{\mathbf{G}})=e^{\ac_{\mathbf{G}}}e^{-\ad_{\mathbf{G}}}\frac{e^{2\ad_{\mathbf{G}}}-1}{2\ad_{\mathbf{G}}}=e^{\ac_{\mathbf{G}}}\frac{\sinh\ad_{\mathbf{G}}}{\ad_{\mathbf{G}}}.
	\end{equation*}
	Variant \eqref{E3} is clearly just a~consequence of \eqref{E4} and \eqref{al2}.
	
	(L) For the derivative of the matrix logarithm, we can proceed completely analogously.	Recalling now also property \eqref{al6}, we get
	\begin{align*}
		\frac{d\log\mathbf{A}}{d\mathbf{A}}=\frac{\log\ml_{\mathbf{A}}-\log\mr_{\mathbf{A}}}{\ml_{\mathbf{A}}-\mr_{\mathbf{A}}}=\frac{\ml_{\log\mathbf{A}}-\mr_{\log\mathbf{A}}}{e^{\ml_{\log\mathbf{A}}}-e^{\mr_{\log\mathbf{A}}}}&=e^{-\mr_{\log\mathbf{A}}}\frac{\ml_{\log\mathbf{A}}-\mr_{\log\mathbf{A}}}{e^{\ml_{\log\mathbf{A}}-\mr_{\log\mathbf{A}}}-1}\\
        &=\mr_{\mathbf{A}^{-1}}(1/\eta)(2\ad_{\log\mathbf{A}}),
	\end{align*}
	which is \eqref{L2}. The other versions \eqref{L1}, \eqref{L3} and \eqref{L4} can be obtained from \eqref{L2} and Lemma~\ref{Had} in a~similar way as for the derivative of the exponential. The inversion relation \eqref{L0} now follows by applying Theorem~\ref{Lrep}~(ii), (iv) to \eqref{E1} (for instance) and using \eqref{funcm} and \eqref{L1}:
	\begin{equation*}
		\Big(\frac{de^{\mathbf{G}}}{d\mathbf{G}}\Big)^{-1}=\frac1{\mr_{e^{\mathbf{G}}}}\frac1{\eta(2\ad_{\mathbf{G}})}=\mr_{e^{-\mathbf{G}}}\frac1{\eta(2\ad_{\mathbf{G}})}=\Big(\frac{d\log\mathbf{A}}{d\mathbf{A}}\Big)_{\mathbf{A}=e^{\mathbf{G}}}.
	\end{equation*}
	In order to verify \eqref{L6}, which is the last remaining identity in (L), let us first calculate that
	\begin{equation*}
		\frac{d}{ds}\frac{\log(1-s+se^x)-\log(1-s+se^y)}{e^x-e^y}=\frac1{(1-s+se^x)(1-s+se^y)},\quad x\neq y\in\mathbb{R},
	\end{equation*}
	hence, by integration of this identity over $(0,1)$, we get
	\begin{equation*}
		\int_0^1\frac{ds}{(1-s+se^x)(1-s+se^y)}=\frac{e^{-y}}{\eta(x-y)},\qquad x,y\in\mathbb{R}.
	\end{equation*}
	With this in hand, it is straightforward to show using \eqref{funcm} and Definition~\ref{Def} that
	\begin{align*}
		\int_0^1\ml_{((1-s)\mathbf{I}+s\mathbf{A})^{-1}}\mr_{((1-s)\mathbf{I}+s\mathbf{A})^{-1}}ds&=\int_0^1\frac{ds}{(1-s+se^{\ml_{\log\mathbf{A}}})(1-s+se^{\mr_{\log\mathbf{A}}})}\\
		&=e^{-\mr_{\log\mathbf{A}}}(1/\eta)(2\ad_{\log\mathbf{A}})=\frac{d\log\mathbf{A}}{d\mathbf{A}},
	\end{align*}
	thanks to \eqref{L2}.
	
	(P) The first identity \eqref{PP0} is again a~direct consequence of definition in \eqref{dfG2}. The second one \eqref{PP00} was already proved in \eqref{wh}. To show \eqref{PP1}, we use the chain rule and previous results (namely \eqref{E4} and \eqref{L4}) to get
	\begin{align}
		\frac{d\mathbf{A}^r}{d\mathbf{A}}=r\Big(\frac{de^{\mathbf{G}}}{d\mathbf{G}}\Big)_{\mathbf{G}=r\log\mathbf{A}}\frac{d\log\mathbf{A}}{d\mathbf{A}}&=e^{\ac_{r\log\mathbf{A}}}\frac{\sinh(r\ad_{\log\mathbf{A}})}{\ad_{\log\mathbf{A}}}e^{-\ac_{\log\mathbf{A}}}\frac{\ad_{\log\mathbf{A}}}{\sinh\ad_{\log\mathbf{A}}}\nonumber\\\
        &=e^{(r-1)\ac_{\log\mathbf{A}}}\frac{\sinh(r\ad_{\log\mathbf{A}})}{\sinh\ad_{\log\mathbf{A}}},\label{comp}
	\end{align}
	which is \eqref{PP2}. In order to get also \eqref{PP1}, we continue in \eqref{comp} and express the first term using \eqref{al2}. This way, applying also the basic identity
	\begin{align*}
		\frac{\sinh(rx)}{\sinh x\cosh((r-1)x)}&=\frac{\sinh x\cosh((r-1)x)+\cosh x\sinh((r-1)x)}{\sinh x\cosh((r-1)x)}\\
        &=1+\frac{\tanh((r-1)x)}{\tanh{x}},
	\end{align*}
	we arrive at
	\begin{equation*}
		\frac{d\mathbf{A}^r}{d\mathbf{A}}
		=\ac_{\mathbf{A}^{r-1}}\frac{\sinh(r\ad_{\log\mathbf{A}})}{\sinh\ad_{\log\mathbf{A}}\cosh((r-1)\ad\log\mathbf{A})}=\ac_{\mathbf{A}^{r-1}}\Big(1+\frac{\tanh{((r-1)\ad_{\log\mathbf{A}})}}{\tanh\ad_{\log\mathbf{A}}}\Big).
	\end{equation*}
	
	(H) All these identities follow immediately from \eqref{E3} and \eqref{al2}.
	
	(G) These identities can be seen as equivalent variants of those in (H), appealing to
	\begin{equation*}
		\cosh x=\cos(ix),\quad\sinh x=-i\sin(ix)
	\end{equation*}
	and the fact that all these functions are entire.
\end{proof}

\end{document}